\documentclass[a4paper,10pt]{amsart}
\usepackage{graphicx}
\usepackage{amsmath}
\usepackage{latexsym}
\usepackage{amsthm}
\usepackage{autograph,epic,latexsym,bezier,amsbsy,color,enumerate,amsfonts,amsmath,amscd,amssymb}
\usepackage[latin1]{inputenc}

\setloopdiam{10}\setprofcurve{7}
\newtheorem{defi}{Definition}[section]
\newtheorem{teo}[defi]{Theorem}

\newtheorem{prop}[defi]{Proposition}
\newtheorem{cor}[defi]{Corollary}

\newtheorem{os}[defi]{Remark}
\def\subjclass#1{{\renewcommand{\thefootnote}{}%
\footnote{\emph{Mathematics Subject Classification (2010):} #1}}}

\begin{document}
\title{Weighted spanning trees on some self-similar graphs}

\author[D. D'Angeli]{Daniele D'Angeli}
\address{%
Daniele D'Angeli\\
Department of Mathematics\\
Technion--Israel Institute of Technology\\
Technion City, Haifa 32 000\\
Israel}\email{dangeli@tx.technion.ac.il}

\author[A. Donno]{Alfredo Donno}
\address{%
Alfredo Donno\\
Dipartimento di Matematica\\
Sapienza Università di Roma\\
Piazzale A. Moro, 2\\
00185 Roma, Italia}\email{donno@mat.uniroma1.it}

\begin{abstract}
We compute the complexity of two infinite families of finite
graphs: the Sierpi\'{n}ski graphs, which are finite approximations
of the well-known Sierpi\'nsky gasket, and the Schreier graphs of
the Hanoi Towers group $H^{(3)}$ acting on the rooted ternary
tree. For both of them, we study the weighted generating functions
of the spanning trees, associated with several natural labellings
of the edge sets.
\end{abstract}

\subjclass{05A15, 05C22, 20E08, 05C25}

\keywords{Spanning tree, weight function, generating function,
self-similar graph, self-similar group, Schreier graph}

\maketitle

\section{Introduction}

The enumeration of spanning trees in a finite graph is largely
studied in the literature, and it has many applications in several
areas of Mathematics as Algebra,
Combinatorics, Probability and of Theoretical Computer Science.\\
\indent Given a connected finite graph $Y=(V(Y),E(Y))$, where
$V(Y)$ and $E(Y)$ denote the vertex set and the edge set of $Y$,
respectively, a \textit{spanning tree} of $Y$ is a subgraph of $Y$
which is a tree and whose vertex set coincides with $V(Y)$.\\
\indent The number of spanning trees of a graph $Y$ is called the
\textit{complexity} of $Y$ and is denoted by $\tau(Y)$. The famous
Kirchhoff's Matrix-Tree Theorem (1847) states that $\tau(Y)$ is
equal to (the constant value of) any cofactor of the Laplace
matrix of $Y$, which is obtained as the difference between the
degree matrix of $Y$ and its adjacency matrix. Equivalently,
$\tau(Y)\cdot|V(Y)|$ is given by the product of all nonzero
eigenvalues of the Laplace matrix of $Y$.\\ \indent It is
interesting to study complexity when the system grows. More
precisely, given a sequence $\{Y_n\}_{n\geq 1}$ of finite graphs
with complexity $\tau(Y_n)$, such that $|V(Y_n)|\to \infty$, the
limit
$$
\lim_{|V(Y_n)| \to \infty}\frac{\log \tau(Y_n)}{|V(Y_n)|},
$$
when it exists, is called the \emph{asymptotic growth
constant} of the spanning trees of $\{Y_n\}_{n\geq 1}$ (see \cite{lyons}).\\
\indent A \textit{spanning $k$-forest} of $Y$ is a subgraph of $Y$
which is a $k$-forest, i.e., it is a forest with $k$ connected
components, and its vertex set coincides with $V(Y)$.\\ \indent
The enumeration of spanning subgraphs, in general, for a graph
$Y$, is also strictly related to the Tutte polynomial $T_Y(x,y)$
of the graph: more precisely, it is known that $T_Y(1,1)$ equals
the complexity of $Y$, $T_Y(2,1)$ equals the number of spanning
forests of $Y$, and $T_Y(1,2)$ is the number of its connected
spanning subgraphs (see \cite{csdi, tutte}, where this analysis is
developed for the finite Sierpi\'{n}ski graphs and for other
examples of finite graphs associated with the action of automorphisms groups of rooted regular trees).\\
\indent A finer invariant of the graph $Y$ is a finite abelian
group $\Phi(Y)$, whose order is exactly the complexity of $Y$.
This group occurs in the literature under different names,
depending on the context. It was introduced in \cite{bacher} as
the Picard group of $Y$ (or the Jacobian of $Y$), whereas it is
shown in \cite{biggs} that the Picard group is isomorphic to the
group of critical configurations of the chip-firing game on $Y$.
As any finite abelian group, $\Phi(Y)$ can be decomposed into
direct sum of invariant factors. The dependence of this
decomposition on the properties of $Y$ has been studied by several
authors, (see, e.g., \cite{lorenzini}), but not much is known so
far. Explicit
computations have been performed for certain families of graphs.\\
\indent In many optimization problems it is often useful to find a
minimal spanning tree of a weighted graph. Hence, it is
interesting to study spanning trees when a weight function on
$E(Y)$ is introduced. So let $w:E(Y)\longrightarrow
\mathbb{R}_{+}$ be a weight function defined on the edges of $Y$.
Let $\mathcal{T}$ be the set of all spanning trees of $Y$. With
each spanning tree $t\in \mathcal{T}$, we can associate the weight
$$
W(t):=\prod_{e\in E(t)} w(e),
$$
i.e., the product of the weights of the edges of $Y$ belonging to
$E(t)$. Then, the \textit{weighted generating function} of the
spanning trees of $Y$ is the polynomial on the variables $w(e), e
\in E(Y)$, given by
$$
T(w):=\sum_{t\in \mathcal{T}} W(t).
$$
It follows from the definition that, when evaluated on the
constant weight $w\equiv 1$, the generating function yields the
complexity of the graph, since in this case one has $W(t)=1$, for
each $t\in \mathcal{T}$.

In this paper, we will study weighted spanning trees on two
infinite families of finite graphs very close to each other: the
Sierpi\'{n}ski graphs, which are finite approximations of the
famous Sierpi\'{n}ski gasket, and the Schreier graphs of the Hanoi
Towers group $H^{(3)}$, which is an example of a {\it self-similar
group} (see Definition \ref{defiselfsimilar} below).

We recall some basic facts about self-similar groups. Let $T_q$ be
the infinite regular rooted tree of degree $q$, i.e., the rooted
tree in which each vertex has $q$ children. Each vertex of the
$n$-th level of the tree can be regarded as a word of length $n$
in the alphabet $X=\{0,1,\ldots, q-1\}$. Now let $G<Aut(T_q)$ be a
group acting on $T_q$ by automorphisms generated by a finite
symmetric set of generators $S$. Suppose, moreover, that the
action is transitive on each level of the tree.

\begin{defi}\label{defischreiernovembre}
The $n$-th {\it Schreier graph} $\Sigma_n$ of the action of $G$ on
$T_q$, with respect to the generating set $S$, is a graph whose
vertex set coincides with the set of vertices of the $n$-th level
of the tree, and two vertices $u,v$ are adjacent if and only if
there exists $s\in S$ such that $s(u)=v$. If this is the case, the
edge joining $u$ and $v$ is labelled by $s$.
\end{defi}
The vertices of $\Sigma_n$ are labelled by words of length $n$ in
$X$ and the edges are labelled by elements of $S$. The Schreier
graph is thus a regular graph of degree $|S|$ with $q^n$ vertices,
and it is connected since the action of $G$ is level-transitive.

\begin{defi}[\cite{volo}]\label{defiselfsimilar}
A finitely generated group $G<Aut(T_q)$ is {\it self-similar} if,
for all $g\in G, x\in X$, there exist $h\in G, y\in X$ such that
$$
g(xw)=yh(w),
$$
for all finite words $w$ in the alphabet $X$.
\end{defi}

Self-similarity implies that $G$ can be embedded into the wreath
product $Sym(q)\wr G = Sym(q)\ltimes G^q$, where $Sym(q)$ denotes
the symmetric group on $q$ elements, so that any automorphism
$g\in G$ can be represented as
$$
g=\alpha(g_0,\ldots,g_{q-1}),
$$
where $\alpha\in Sym(q)$ describes the action of $g$ on the first
level of $T_q$ and $g_i\in G, i=0,...,q-1$, is the restriction of
$g$ on the full subtree of $T_q$ rooted at the vertex $i$ of the
first level of $T_q$ (observe that any such subtree is isomorphic
to $T_q$). Hence, if $x\in X$ and $w$ is a finite word in $X$, we
have
$$
g(xw)=\alpha(x)g_x(w).
$$
\indent The class of self-similar groups contains many interesting
examples of groups which have exotic properties: among them, we
mention the first Grigorchuk group, which yields the simplest
solution of the Burnside problem (an infinite, finitely generated
torsion group) and the first example of a group of intermediate
growth (see \cite{grigorchuk} for a detailed account and further
references). In the last decades, the study of automorphisms
groups of rooted trees has been largely investigated: R.
Grigorchuk and a number of coauthors have developed a new exciting
direction of research focusing on finitely generated groups acting
by automorphisms on rooted trees \cite{fractal}. They proved that
these groups have deep connections with the theory of profinite
groups and with complex dynamics. In particular, for many examples
of groups belonging to this class, the property of self-similarity
is reflected on fractalness of some limit objects associated with
them \cite{volo}.

Since the Schreier graphs are determined by group actions, their
edges are naturally labelled by the generators of the acting group
and it takes sense to study weighted spanning trees on them, with
respect to this labelling.

The paper is structured as follows. In Section
\ref{sectionsierpinski}, we study weighted spanning trees on
finite approximations of the well-known Sierpi\'{n}ski gasket,
endowed with three different edge labellings:
\begin{itemize}
\item the \lq\lq rotational-invariant\rq\rq labelling, whose special
symmetry allows to explicitly compute the generating function of
the spanning trees (Theorem \ref{modellofacileteo}) and to perform
a statistical analysis about the number of edges, with a fixed
label, occurring in a random spanning tree of the graph
(Proposition \ref{propstat});
\item the \lq\lq directional\rq\rq labelling, where the weights depend
on the direction of the edges; for this model, the weighted
generating function of the spanning trees is described via the
iteration of a polynomial map (Theorem \ref{noname});
\item the \lq\lq Schreier\rq\rq labelling, strictly related to the
labelling of the Schreier graphs of the Hanoi Towers group
$H^{(3)}$; also in this case, the weighted generating function of
the spanning trees is described via the iteration of a polynomial
map (Theorem \ref{noname2}).
\end{itemize}

In all these models we follow a combinatorial approach. The
self-similar structure of the graph (in the sense of
\cite{wagnerself}) allows to study both unweighted and weighted
subgraphs recursively. More precisely, we introduce three
different generating functions associated with spanning trees,
$2$-spanning forests, $3$-spanning forests and, using
self-similarity, we are able to establish recursive relations
(Theorems \ref{equazionimodellofacile},
\ref{geometrichedirezionale} and \ref{equazionisoddisfatte}) and
to give an explicit description of them (Theorems
\ref{modellofacileteo}, \ref{noname} and \ref{noname2}). More
generally, the self-similar structure of a graph turns out to be a
powerful tool for investigating many combinatorial and statistical
models on it: see, for instance, \cite{noiising, noidimeri,
wagner1, wagnerself, wagneraltro}.

In Section \ref{sectionhanoi}, we consider the Schreier graphs of
the Hanoi Towers group $H^{(3)}$, whose action on the ternary tree
models the famous Hanoi Towers game on three pegs (see
\cite{hanoi}), endowed with the natural edge labelling coming from
the action of its generators. Even if these graphs also have a
self-similar structure, the combinatorial approach used in the
case of the Sierpi\'{n}ski graphs seems to be much harder here.
Therefore, our technique consists in using a weighted version of
the Kirchhoff's Theorem: we construct the Laplace matrix by using
the self-similar presentation of the generators of the group,
which is impossible in the case of Sierpi\'{n}ski graphs, where
there is no group structure. In this case, the generating function
is described in terms of iterations of a rational map (Theorem
\ref{PROPOSITIONPARTITION}): this kind of approach already appears
in \cite{hecketype, hanoi} (see also \cite{noidimeri}, where we
use the same strategy to compute the partition function of the
dimer model on the Schreier graphs of the Hanoi Towers group).


\section{Spanning trees on the Sierpi\'{n}ski graphs}\label{sectionsierpinski}

The problem of enumeration of spanning trees in Sierpi\'{n}ski
graphs was largely treated in literature (see, for instance,
\cite{taiwan, wagner1}). We consider here three different
labellings of the edges of these graphs and write down the
associated generating function of the spanning trees. In all the
models, the self-similarity of the graphs plays a crucial role to
study the problem recursively. The description of the generating
function strongly depends on the symmetry of the labelling of the
graph: as we will see, in the first model that we consider, which
is invariant under rotation, we are able to give an explicit
formula for it; in the two remaining models, where we do not have
invariance under the action of any symmetry group, the generating
function is described via the iteration of two polynomial maps.

\newpage
\subsection{First model: \lq\lq Rotational-invariant\rq\rq labelling}\label{rotational}

Let $\Gamma_1$ be the graph in the following
picture.\unitlength=0,2mm

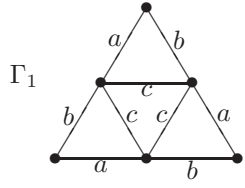
\begin{center}
\begin{picture}(400,120)
\put(110,60){$\Gamma_1$}

\letvertex D=(200,110)\letvertex E=(170,60)\letvertex F=(140,10)\letvertex G=(200,10)
\letvertex H=(260,10)\letvertex I=(230,60)

\drawvertex(D){$\bullet$}
\drawvertex(E){$\bullet$}\drawvertex(F){$\bullet$}
\drawvertex(G){$\bullet$}\drawvertex(H){$\bullet$}
\drawvertex(I){$\bullet$}

\drawundirectededge(E,D){$a$} \drawundirectededge(F,E){$b$}
\drawundirectededge(G,F){$a$}

\drawundirectededge(H,G){$b$} \drawundirectededge(I,H){$a$}
\drawundirectededge(D,I){$b$} \drawundirectededge(I,E){$c$}
\drawundirectededge(E,G){$c$} \drawundirectededge(G,I){$c$}
\end{picture}
\end{center}
For each $n\geq 1$, we define, by recurrence, the graph
$\Gamma_{n+1}$ as the graph obtained by partitioning an
equilateral triangle in four smaller equilateral triangles and by
putting in each corner a copy of $\Gamma_{n}$. Observe that this
labelling of the graph is invariant with respect to the rotation
of $\frac{2\pi}{3}$. We represent in the following picture the
graph $\Gamma_2$.\unitlength=0.2mm
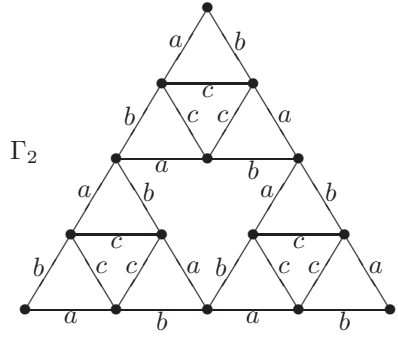
\begin{center}
\begin{picture}(400,220)
\put(70,110){$\Gamma_2$}

\letvertex A=(200,210)\letvertex B=(170,160)\letvertex C=(140,110)

\letvertex D=(110,60)\letvertex E=(80,10)\letvertex F=(140,10)\letvertex G=(200,10)

\letvertex H=(260,10)\letvertex I=(320,10)
\letvertex L=(290,60)\letvertex M=(260,110)\letvertex N=(230,160)

\letvertex O=(200,110)\letvertex P=(170,60)\letvertex Q=(230,60)

\drawvertex(A){$\bullet$}\drawvertex(B){$\bullet$}
\drawvertex(C){$\bullet$}\drawvertex(D){$\bullet$}
\drawvertex(E){$\bullet$}\drawvertex(F){$\bullet$}
\drawvertex(G){$\bullet$}\drawvertex(H){$\bullet$}
\drawvertex(I){$\bullet$}\drawvertex(L){$\bullet$}\drawvertex(M){$\bullet$}
\drawvertex(N){$\bullet$}\drawvertex(O){$\bullet$}
\drawvertex(P){$\bullet$}\drawvertex(Q){$\bullet$}

\drawundirectededge(E,D){$b$} \drawundirectededge(D,C){$a$}
\drawundirectededge(C,B){$b$} \drawundirectededge(B,A){$a$}
\drawundirectededge(A,N){$b$} \drawundirectededge(N,M){$a$}
\drawundirectededge(M,L){$b$} \drawundirectededge(L,I){$a$}
\drawundirectededge(I,H){$b$} \drawundirectededge(H,G){$a$}
\drawundirectededge(G,F){$b$} \drawundirectededge(F,E){$a$}
\drawundirectededge(N,B){$c$} \drawundirectededge(O,C){$a$}
\drawundirectededge(M,O){$b$} \drawundirectededge(P,D){$c$}
\drawundirectededge(L,Q){$c$} \drawundirectededge(B,O){$c$}
\drawundirectededge(O,N){$c$} \drawundirectededge(C,P){$b$}
\drawundirectededge(P,G){$a$} \drawundirectededge(D,F){$c$}
\drawundirectededge(Q,M){$a$} \drawundirectededge(G,Q){$b$}
\drawundirectededge(H,L){$c$}\drawundirectededge(F,P){$c$}
\drawundirectededge(Q,H){$c$}
\end{picture}
\end{center}

We want to study weighted spanning trees on the graphs
$\{\Gamma_n\}_{n\geq 1}$. For each $n\geq 1$, we put:
\begin{itemize}
\item $T_n(a,b,c)=$ weighted generating function of the spanning trees of
$\Gamma_n$;
\item $S_n(a,b,c)=$ weighted generating function of the spanning $2$-forests of $\Gamma_n$, where two fixed outmost vertices belong to the same connected
component and the third outmost vertex belongs to the second
connected component;
\item $Q_n(a,b,c)=$ weighted generating function of the spanning $3$-forests of $\Gamma_n$, where the three outmost vertices belong to three different
connected components.
\end{itemize}
Observe that, because of the rotational invariance of the
labelling of the graph, the function $S_n(a,b,c)$ does not depend
on the choice of the two outmost vertices. In what follows, we
will often omit the argument $(a,b,c)$ of the weighted generating
functions.

\begin{teo}\label{equazionimodellofacile}
For each $n\geq 1$, the weighted generating functions $T_n(a,b,c),
S_n(a,b,c)$ and $Q_n(a,b,c)$ satisfy the following equations:
\begin{eqnarray}\label{Tfacile}
T_{n+1}=6T_{n}^2S_{n}
\end{eqnarray}
\begin{eqnarray}\label{Sfacile}
S_{n+1}=7T_{n}S_{n}^2 + T_{n}^2Q_{n}
\end{eqnarray}
\begin{eqnarray}\label{Qfacile}
Q_{n+1}=12T_{n}S_{n}Q_{n} +14S_{n}^3,
\end{eqnarray}
with initial conditions
$$
T_1(a,b,c) = 3(a+b)(ab+ac+bc)^2
$$
$$
S_1(a,b,c) = (a+b)(a+b+3c)(ab+ac+bc) \qquad Q_1(a,b,c) =
(a+b)(a+b+3c)^2.
$$
\end{teo}

\begin{proof}
The graph $\Gamma_{n+1}$ can be represented as a triangle
containing three copies of $\Gamma_{n}$.
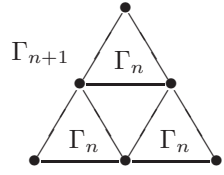
\begin{center}\unitlength=0.3mm
\begin{picture}(400,85)
\letvertex A=(200,78)\letvertex B=(180,44)
\letvertex C=(160,10)\letvertex D=(200,10)
\letvertex E=(240,10)\letvertex F=(220,44)

\put(150,55){$\Gamma_{n+1}$} \put(195,51){$\Gamma_{n}$}
\put(175,16){$\Gamma_{n}$} \put(215,16){$\Gamma_{n}$}

\drawvertex(A){$\bullet$}\drawvertex(B){$\bullet$}
\drawvertex(C){$\bullet$}\drawvertex(D){$\bullet$}
\drawvertex(E){$\bullet$}\drawvertex(F){$\bullet$}

\drawundirectededge(D,B){} \drawundirectededge(F,D){}
\drawundirectededge(B,F){} \drawundirectededge(D,E){}
\drawundirectededge(B,C){}\drawundirectededge(F,A){}\drawundirectededge(A,B){}
\drawundirectededge(C,D){} \drawundirectededge(E,F){}
\end{picture}
\end{center}

\noindent We will use the pictures
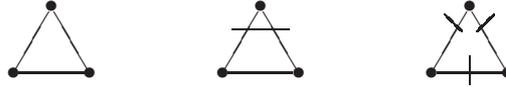
\begin{center}
\begin{picture}(500,55)
\unitlength=0.25mm
\letvertex A=(90,45)\letvertex B=(70,10)\letvertex C=(110,10)

\letvertex D=(200,45)\letvertex E=(180,10)\letvertex F=(220,10)

\letvertex G=(310,45)\letvertex H=(290,10)\letvertex I=(330,10)

\letvertex a=(182,33)\letvertex b=(218,33)

\letvertex c=(295,43)\letvertex d=(308,30)\letvertex e=(325,43)

\letvertex f=(312,30)\letvertex g=(310,22)\letvertex h=(310,1)

\drawvertex(A){$\bullet$}\drawvertex(B){$\bullet$}
\drawvertex(C){$\bullet$}\drawvertex(D){$\bullet$}
\drawvertex(E){$\bullet$}\drawvertex(F){$\bullet$}\drawvertex(G){$\bullet$}
\drawvertex(H){$\bullet$}\drawvertex(I){$\bullet$}

\drawundirectededge(A,B){} \drawundirectededge(B,C){}
\drawundirectededge(C,A){} \drawundirectededge(D,E){}
\drawundirectededge(E,F){}\drawundirectededge(F,D){}\drawundirectededge(G,H){}
\drawundirectededge(H,I){} \drawundirectededge(I,G){}

\drawundirectededge(a,b){} \drawundirectededge(c,d){}
\drawundirectededge(e,f){} \drawundirectededge(h,g){}
\end{picture}
\end{center}
to denote, respectively, the case where in a copy of $\Gamma_{n}$:
\begin{itemize}
\item the three outmost vertices are in the same connected
component;
\item two outmost vertices are in the same connected component and
the third one is in a different connected component;
\item the outmost vertices are in three different connected
components.
\end{itemize}

The only way to construct a spanning tree of $\Gamma_{n+1}$ is to
choose a spanning tree in two copies of $\Gamma_{n}$ and a
spanning $2$-forest in the third one, as in the following picture.
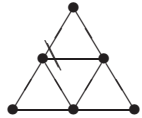
\begin{center}
\begin{picture}(400,80)
\letvertex A=(200,78)\letvertex B=(180,44)
\letvertex C=(160,10)\letvertex D=(200,10)
\letvertex E=(240,10)\letvertex F=(220,44)

\letvertex x=(180,59)\letvertex X=(193,34)

\drawundirectededge(x,X){}

\drawvertex(A){$\bullet$}\drawvertex(B){$\bullet$}
\drawvertex(C){$\bullet$}\drawvertex(D){$\bullet$}
\drawvertex(E){$\bullet$}\drawvertex(F){$\bullet$}

\drawundirectededge(D,B){} \drawundirectededge(F,D){}
\drawundirectededge(B,F){} \drawundirectededge(D,E){}
\drawundirectededge(B,C){}\drawundirectededge(F,A){}\drawundirectededge(A,B){}
\drawundirectededge(C,D){} \drawundirectededge(E,F){}
\end{picture}
\end{center}

This argument proves Equation \eqref{Tfacile}, where the factor
$6$ is given by symmetry (we have to take into account both
reflections and rotations).

Next, we are going to prove Equation \eqref{Sfacile} (we analyze,
for instance, the case where the leftmost and the rightmost
vertices are in the same connected component). Consider the two
following pictures.

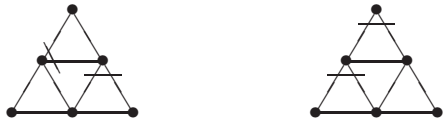
\begin{center}
\begin{picture}(400,80)
\letvertex A=(100,78)\letvertex B=(80,44)
\letvertex C=(60,10)\letvertex D=(100,10)
\letvertex E=(140,10)\letvertex F=(120,44)

\letvertex x=(80,59)\letvertex X=(93,34)
\letvertex y=(105,35)\letvertex Y=(135,35)
\drawundirectededge(x,X){}\drawundirectededge(y,Y){}

\letvertex u=(285,69)\letvertex U=(315,69)
\letvertex v=(265,35)\letvertex V=(295,35)

\drawundirectededge(u,U){}\drawundirectededge(v,V){}

\drawvertex(A){$\bullet$}\drawvertex(B){$\bullet$}
\drawvertex(C){$\bullet$}\drawvertex(D){$\bullet$}
\drawvertex(E){$\bullet$}\drawvertex(F){$\bullet$}

\drawundirectededge(D,B){} \drawundirectededge(F,D){}
\drawundirectededge(B,F){} \drawundirectededge(D,E){}
\drawundirectededge(B,C){}\drawundirectededge(F,A){}\drawundirectededge(A,B){}
\drawundirectededge(C,D){} \drawundirectededge(E,F){}

\letvertex a=(300,78)\letvertex b=(280,44)
\letvertex c=(260,10)\letvertex d=(300,10)
\letvertex e=(340,10)\letvertex f=(320,44)

\drawvertex(a){$\bullet$}\drawvertex(b){$\bullet$}
\drawvertex(c){$\bullet$}\drawvertex(d){$\bullet$}
\drawvertex(e){$\bullet$}\drawvertex(f){$\bullet$}

\drawundirectededge(d,b){} \drawundirectededge(f,d){}
\drawundirectededge(b,f){} \drawundirectededge(d,e){}
\drawundirectededge(b,c){}\drawundirectededge(f,a){}\drawundirectededge(a,b){}
\drawundirectededge(c,d){} \drawundirectededge(e,f){}
\end{picture}
\end{center}
These possibilities, together with their symmetric, obtained by
reflecting with respect to the vertical axis, give a contribution
to $S_{n+1}$ equal to $4T_{n}S_{n}^2$. Consider now the two
following configurations.

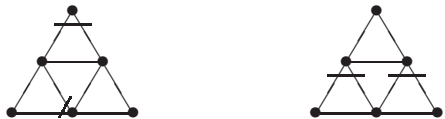
\begin{center}
\begin{picture}(400,80)
\letvertex A=(100,78)\letvertex B=(80,44)
\letvertex C=(60,10)\letvertex D=(100,10)
\letvertex E=(140,10)\letvertex F=(120,44)

\letvertex x=(85,69)\letvertex X=(115,69)
\letvertex y=(100,23)\letvertex Y=(90,5)
\drawundirectededge(x,X){}\drawundirectededge(y,Y){}

\letvertex u=(265,35)\letvertex U=(295,35)
\letvertex v=(305,35)\letvertex V=(335,35)

\drawundirectededge(u,U){}\drawundirectededge(v,V){}

\drawvertex(A){$\bullet$}\drawvertex(B){$\bullet$}
\drawvertex(C){$\bullet$}\drawvertex(D){$\bullet$}
\drawvertex(E){$\bullet$}\drawvertex(F){$\bullet$}

\drawundirectededge(D,B){} \drawundirectededge(F,D){}
\drawundirectededge(B,F){} \drawundirectededge(D,E){}
\drawundirectededge(B,C){}\drawundirectededge(F,A){}\drawundirectededge(A,B){}
\drawundirectededge(C,D){} \drawundirectededge(E,F){}

\letvertex a=(300,78)\letvertex b=(280,44)
\letvertex c=(260,10)\letvertex d=(300,10)
\letvertex e=(340,10)\letvertex f=(320,44)

\drawvertex(a){$\bullet$}\drawvertex(b){$\bullet$}
\drawvertex(c){$\bullet$}\drawvertex(d){$\bullet$}
\drawvertex(e){$\bullet$}\drawvertex(f){$\bullet$}

\drawundirectededge(d,b){} \drawundirectededge(f,d){}
\drawundirectededge(b,f){} \drawundirectededge(d,e){}
\drawundirectededge(b,c){}\drawundirectededge(f,a){}\drawundirectededge(a,b){}
\drawundirectededge(c,d){} \drawundirectededge(e,f){}
\end{picture}
\end{center}
Since the picture on the left has to be considered together with
its symmetric, we get a contribution to $S_{n+1}$ equal to
$3T_{n}S_{n}^2$. Finally, the contribution $T_{n}^2Q_{n}$ is
described by the following picture.
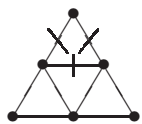
\begin{center}
\begin{picture}(400,85)
\letvertex A=(200,78)\letvertex B=(180,44)
\letvertex C=(160,10)\letvertex D=(200,10)
\letvertex E=(240,10)\letvertex F=(220,44)

\letvertex a=(182,70)\letvertex b=(198,50)
\letvertex c=(218,70)\letvertex d=(202,50)
\letvertex e=(200,54)\letvertex f=(200,34)

\drawundirectededge(a,b){} \drawundirectededge(c,d){}
\drawundirectededge(e,f){}

\drawvertex(A){$\bullet$}\drawvertex(B){$\bullet$}
\drawvertex(C){$\bullet$}\drawvertex(D){$\bullet$}
\drawvertex(E){$\bullet$}\drawvertex(F){$\bullet$}

\drawundirectededge(D,B){} \drawundirectededge(F,D){}
\drawundirectededge(B,F){} \drawundirectededge(D,E){}
\drawundirectededge(B,C){}\drawundirectededge(F,A){}\drawundirectededge(A,B){}
\drawundirectededge(C,D){} \drawundirectededge(E,F){}
\end{picture}
\end{center}
This proves Equation \eqref{Sfacile}.

We have now to prove Equation \eqref{Qfacile} about $Q_{n+1}$.
Consider the following situations.
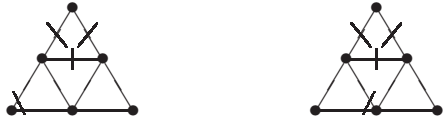
\begin{center}
\begin{picture}(400,80)
\letvertex A=(100,78)\letvertex B=(80,44)
\letvertex C=(60,10)\letvertex D=(100,10)
\letvertex E=(140,10)\letvertex F=(120,44)

\letvertex aa=(82,70)\letvertex bb=(98,50)
\letvertex cc=(118,70)\letvertex dd=(102,50)
\letvertex ee=(100,54)\letvertex ff=(100,34)
\drawundirectededge(aa,bb){} \drawundirectededge(cc,dd){}
\drawundirectededge(ee,ff){}

\letvertex aaa=(282,70)\letvertex bbb=(298,50)
\letvertex ccc=(318,70)\letvertex ddd=(302,50)
\letvertex eee=(300,54)\letvertex fff=(300,34)
\drawundirectededge(aaa,bbb){} \drawundirectededge(ccc,ddd){}
\drawundirectededge(eee,fff){}

\letvertex AA=(60,25)\letvertex BB=(70,5)
\letvertex CC=(300,25)\letvertex DD=(290,5)
\drawundirectededge(AA,BB){} \drawundirectededge(CC,DD){}

\drawvertex(A){$\bullet$}\drawvertex(B){$\bullet$}
\drawvertex(C){$\bullet$}\drawvertex(D){$\bullet$}
\drawvertex(E){$\bullet$}\drawvertex(F){$\bullet$}

\drawundirectededge(D,B){} \drawundirectededge(F,D){}
\drawundirectededge(B,F){} \drawundirectededge(D,E){}
\drawundirectededge(B,C){}\drawundirectededge(F,A){}\drawundirectededge(A,B){}
\drawundirectededge(C,D){} \drawundirectededge(E,F){}

\letvertex a=(300,78)\letvertex b=(280,44)
\letvertex c=(260,10)\letvertex d=(300,10)
\letvertex e=(340,10)\letvertex f=(320,44)

\drawvertex(a){$\bullet$}\drawvertex(b){$\bullet$}
\drawvertex(c){$\bullet$}\drawvertex(d){$\bullet$}
\drawvertex(e){$\bullet$}\drawvertex(f){$\bullet$}

\drawundirectededge(d,b){} \drawundirectededge(f,d){}
\drawundirectededge(b,f){} \drawundirectededge(d,e){}
\drawundirectededge(b,c){}\drawundirectededge(f,a){}\drawundirectededge(a,b){}
\drawundirectededge(c,d){} \drawundirectededge(e,f){}
\end{picture}
\end{center}
They provide, by symmetry, a contribution equal to
$12T_{n}S_{n}Q_{n}$. The following pictures give, by symmetry, a
contribution of $12S_{n}^3$ to $Q_{n+1}$.

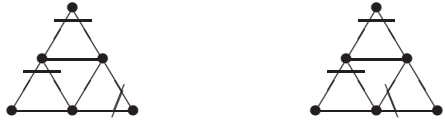
\begin{center}
\begin{picture}(400,80)
\letvertex A=(100,78)\letvertex B=(80,44)
\letvertex C=(60,10)\letvertex D=(100,10)
\letvertex E=(140,10)\letvertex F=(120,44)

\letvertex AA=(85,70)\letvertex BB=(115,70)
\letvertex CC=(285,70)\letvertex DD=(315,70)
\drawundirectededge(AA,BB){} \drawundirectededge(CC,DD){}

\letvertex AAA=(65,36)\letvertex BBB=(95,36)
\letvertex CCC=(265,36)\letvertex DDD=(295,36)
\drawundirectededge(AAA,BBB){} \drawundirectededge(CCC,DDD){}

\letvertex u=(135,30)\letvertex U=(125,3)
\letvertex v=(305,30)\letvertex V=(315,3)
\drawundirectededge(u,U){} \drawundirectededge(v,V){}

\drawvertex(A){$\bullet$}\drawvertex(B){$\bullet$}
\drawvertex(C){$\bullet$}\drawvertex(D){$\bullet$}
\drawvertex(E){$\bullet$}\drawvertex(F){$\bullet$}

\drawundirectededge(D,B){} \drawundirectededge(F,D){}
\drawundirectededge(B,F){} \drawundirectededge(D,E){}
\drawundirectededge(B,C){}\drawundirectededge(F,A){}\drawundirectededge(A,B){}
\drawundirectededge(C,D){} \drawundirectededge(E,F){}

\letvertex a=(300,78)\letvertex b=(280,44)
\letvertex c=(260,10)\letvertex d=(300,10)
\letvertex e=(340,10)\letvertex f=(320,44)

\drawvertex(a){$\bullet$}\drawvertex(b){$\bullet$}
\drawvertex(c){$\bullet$}\drawvertex(d){$\bullet$}
\drawvertex(e){$\bullet$}\drawvertex(f){$\bullet$}

\drawundirectededge(d,b){} \drawundirectededge(f,d){}
\drawundirectededge(b,f){} \drawundirectededge(d,e){}
\drawundirectededge(b,c){}\drawundirectededge(f,a){}\drawundirectededge(a,b){}
\drawundirectededge(c,d){} \drawundirectededge(e,f){}
\end{picture}
\end{center}
Finally, the following two pictures give a contribution of
$2S_{n}^3$ to $Q_{n+1}$.
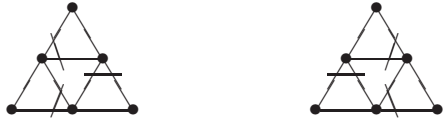
\begin{center}
\begin{picture}(400,80)
\letvertex A=(100,78)\letvertex B=(80,44)
\letvertex C=(60,10)\letvertex D=(100,10)
\letvertex E=(140,10)\letvertex F=(120,44)

\letvertex AA=(105,34)\letvertex BB=(135,34)
\letvertex CC=(265,34)\letvertex DD=(295,34)
\drawundirectededge(AA,BB){} \drawundirectededge(CC,DD){}

\letvertex AAA=(85,64)\letvertex BBB=(95,34)
\letvertex CCC=(315,64)\letvertex DDD=(305,34)
\drawundirectededge(AAA,BBB){} \drawundirectededge(CCC,DDD){}

\letvertex u=(95,30)\letvertex U=(85,3)
\letvertex v=(305,30)\letvertex V=(315,3)
\drawundirectededge(u,U){} \drawundirectededge(v,V){}

\drawvertex(A){$\bullet$}\drawvertex(B){$\bullet$}
\drawvertex(C){$\bullet$}\drawvertex(D){$\bullet$}
\drawvertex(E){$\bullet$}\drawvertex(F){$\bullet$}

\drawundirectededge(D,B){} \drawundirectededge(F,D){}
\drawundirectededge(B,F){} \drawundirectededge(D,E){}
\drawundirectededge(B,C){}\drawundirectededge(F,A){}\drawundirectededge(A,B){}
\drawundirectededge(C,D){} \drawundirectededge(E,F){}

\letvertex a=(300,78)\letvertex b=(280,44)
\letvertex c=(260,10)\letvertex d=(300,10)
\letvertex e=(340,10)\letvertex f=(320,44)

\drawvertex(a){$\bullet$}\drawvertex(b){$\bullet$}
\drawvertex(c){$\bullet$}\drawvertex(d){$\bullet$}
\drawvertex(e){$\bullet$}\drawvertex(f){$\bullet$}

\drawundirectededge(d,b){} \drawundirectededge(f,d){}
\drawundirectededge(b,f){} \drawundirectededge(d,e){}
\drawundirectededge(b,c){}\drawundirectededge(f,a){}\drawundirectededge(a,b){}
\drawundirectededge(c,d){} \drawundirectededge(e,f){}
\end{picture}
\end{center}
This completes the proof.
\end{proof}

\begin{teo}\label{modellofacileteo}
For each $n\geq 1$, the weighted generating functions $T_n(a,b,c),
S_n(a,b,c)$ and $Q_n(a,b,c)$ satisfying Equations \eqref{Tfacile},
\eqref{Sfacile} and \eqref{Qfacile}, with the initial conditions
given in Theorem \ref{equazionimodellofacile}, are:
$$
T_n(a,b,c)=2^{\frac{3^{n-1}-1}{2}}3^{\frac{3^n+2n-1}{4}}5^{\frac{3^{n-1}-2n+1}{4}}(a+b)^{3^{n-1}}(a+b+3c)^{\frac{3^{n-1}-1}{2}}
(ab+ac+bc)^{\frac{3^n+1}{2}};
$$
$$
S_n(a,b,c)=2^{\frac{3^{n-1}-1}{2}}3^{\frac{3^n-2n-1}{4}}5^{\frac{3^{n-1}+2n-3}{4}}(a+b)^{3^{n-1}}(a+b+3c)^{\frac{3^{n-1}+1}{2}}
(ab+ac+bc)^{\frac{3^n-1}{2}};
$$
$$
Q_n(a,b,c)=2^{\frac{3^{n-1}-1}{2}}3^{\frac{3^n-6n+3}{4}}5^{\frac{3^{n-1}+6n-7}{4}}(a+b)^{3^{n-1}}(a+b+3c)^{\frac{3^{n-1}+3}{2}}
(ab+ac+bc)^{\frac{3^n-3}{2}}.
$$
\end{teo}

\begin{proof}
The proof is by induction on $n$. It is easy to verify that, for
$n=1$, one gets the initial conditions given in Theorem
\ref{equazionimodellofacile}. Then, one can check that the
functions given in the claim satisfy Equations \eqref{Tfacile},
\eqref{Sfacile} and \eqref{Qfacile}. We omit here the explicit
computations.
\end{proof}

It follows that $T_n(1,1,1)=\tau(\Gamma_n)$; similarly,
$s(\Gamma_n):= S_n(1,1,1)$ is the number of spanning $2$-forests
of $\Gamma_n$, where two fixed outmost vertices belong to the same
connected component and the third outmost vertex belongs to the
second connected component; $q(\Gamma_n):=Q_n(1,1,1)$ is the
number of spanning $3$-forests of $\Gamma_n$, where the three
outmost vertices belong to three different connected components.

\begin{cor}
For each $n\geq 1$, one has:
\begin{enumerate}
\item $\displaystyle \tau(\Gamma_n)=
2^{\frac{3^n-1}{2}}3^{\frac{3^{n+1}+2n+1}{4}}5^{\frac{3^n-2n-1}{4}}$;
\item $\displaystyle s(\Gamma_n)=
2^{\frac{3^n-1}{2}}3^{\frac{3^{n+1}-2n-3}{4}}5^{\frac{3^n+2n-1}{4}}$;
\item $\displaystyle q(\Gamma_n)= 2^{\frac{3^n-1}{2}}3^{\frac{3^{n+1}-6n-3}{4}}5^{\frac{3^n+6n-1}{4}}$.
\end{enumerate}
In particular, the asymptotic growth constant of the spanning
trees of $\Gamma_n$ is $\frac{1}{3}\log 2+\frac{1}{2}\log
3+\frac{1}{6}\log 5$.
\end{cor}

\begin{proof}
It suffices to evaluate the weighted generating functions
described in Theorem \ref{modellofacileteo} for $a=b=c=1$. The
asymptotic growth constant is then obtained as the limit
$$
\lim_{n\to \infty}\frac{\log(\tau(\Gamma_n))}{|V(\Gamma_n)|},
$$
where $|V(\Gamma_n)| = \frac{3}{2}(3^n+1)$ is the number of
vertices of $\Gamma_n$, for each $n\geq 1$.
\end{proof}

\begin{os}\rm
The same values of the complexity and of the asymptotic growth
constant have been found in \cite{taiwan} and \cite{wagner1},
where the authors study unweighted spanning trees of $\Gamma_n$.
\end{os}


\subsection{Second model: \lq\lq directional\rq\rq labelling}\label{directional}

Consider now a new sequence of graphs $\{\Gamma_n\}_{n\geq 1}$,
which coincide, as unweighted graphs, with the graphs studied in
Section \ref{rotational}, and whose edges are endowed with a new
labelling, that we call directional labelling. It is clear that an
edge of $\Gamma_n$ can point in three different directions: up
(from left to right), down (from left to right) or horizontal.
Then, we label by $a$ each edge pointing up, by $b$ each
horizontal edge, and by $c$ each edge pointing down, where, as
usual, $a,b,c\in \mathbb{R}_+$. Here we draw the three first
examples.\unitlength=0,18mm
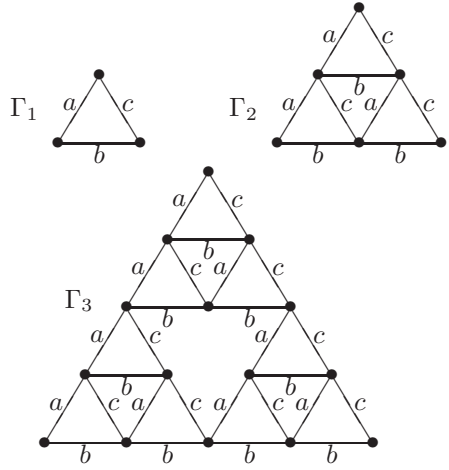
\begin{center}
\begin{picture}(400,100)

\put(55,30){$\Gamma_1$}\put(215,30){$\Gamma_2$}

\letvertex A=(120,60)\letvertex B=(90,10)\letvertex C=(150,10)

\letvertex D=(310,110)\letvertex E=(280,60)\letvertex F=(250,10)\letvertex G=(310,10)
\letvertex H=(370,10)\letvertex I=(340,60)

\drawvertex(A){$\bullet$}\drawvertex(B){$\bullet$}
\drawvertex(C){$\bullet$}\drawvertex(D){$\bullet$}
\drawvertex(E){$\bullet$}\drawvertex(F){$\bullet$}
\drawvertex(G){$\bullet$}\drawvertex(H){$\bullet$}
\drawvertex(I){$\bullet$}

\drawundirectededge(B,A){$a$} \drawundirectededge(C,B){$b$}
\drawundirectededge(A,C){$c$} \drawundirectededge(E,D){$a$}
\drawundirectededge(F,E){$a$} \drawundirectededge(G,F){$b$}

\drawundirectededge(H,G){$b$} \drawundirectededge(I,H){$c$}
\drawundirectededge(D,I){$c$} \drawundirectededge(I,E){$b$}
\drawundirectededge(E,G){$c$} \drawundirectededge(G,I){$a$}
\end{picture}
\end{center}
\begin{center}
\begin{picture}(400,200)

\put(95,110){$\Gamma_3$}

\letvertex A=(200,210)\letvertex B=(170,160)\letvertex C=(140,110)
\letvertex D=(110,60)\letvertex E=(80,10)\letvertex F=(140,10)\letvertex G=(200,10)
\letvertex H=(260,10)\letvertex I=(320,10)
\letvertex L=(290,60)\letvertex M=(260,110)\letvertex N=(230,160)
\letvertex O=(200,110)\letvertex P=(170,60)\letvertex Q=(230,60)

\drawvertex(A){$\bullet$}\drawvertex(B){$\bullet$}
\drawvertex(C){$\bullet$}\drawvertex(D){$\bullet$}
\drawvertex(E){$\bullet$}\drawvertex(F){$\bullet$}
\drawvertex(G){$\bullet$}\drawvertex(H){$\bullet$}
\drawvertex(I){$\bullet$}\drawvertex(L){$\bullet$}\drawvertex(M){$\bullet$}
\drawvertex(N){$\bullet$}\drawvertex(O){$\bullet$}
\drawvertex(P){$\bullet$}\drawvertex(Q){$\bullet$}

\drawundirectededge(E,D){$a$} \drawundirectededge(D,C){$a$}
\drawundirectededge(C,B){$a$} \drawundirectededge(B,A){$a$}
\drawundirectededge(A,N){$c$} \drawundirectededge(N,M){$c$}
\drawundirectededge(M,L){$c$} \drawundirectededge(L,I){$c$}
\drawundirectededge(I,H){$b$} \drawundirectededge(H,G){$b$}
\drawundirectededge(G,F){$b$} \drawundirectededge(F,E){$b$}
\drawundirectededge(N,B){$b$} \drawundirectededge(O,C){$b$}
\drawundirectededge(M,O){$b$} \drawundirectededge(P,D){$b$}
\drawundirectededge(L,Q){$b$} \drawundirectededge(B,O){$c$}
\drawundirectededge(O,N){$a$} \drawundirectededge(C,P){$c$}
\drawundirectededge(P,G){$c$} \drawundirectededge(D,F){$c$}
\drawundirectededge(Q,M){$a$} \drawundirectededge(G,Q){$a$}
\drawundirectededge(H,L){$a$} \drawundirectededge(F,P){$a$}
\drawundirectededge(Q,H){$c$}
\end{picture}
\end{center}
\begin{os}\rm
Observe that the indices are now shifted by $1$ with respect to
the case of the rotational-invariant labelling considered in
Section \ref{rotational}: the reason is that a
rotational-invariant labelling using three labels $a,b$ and $c$
cannot be defined on a simple triangle.
\end{os}

In this section, we study the weighted spanning trees of the graph
$\Gamma_n$ endowed with the directional labelling. For each $n\geq
1$, we put:
\begin{itemize}
\item $T_n(a,b,c)=$ weighted generating function of the spanning trees of
$\Gamma_n$;
\item $U_n(a,b,c)=$ weighted generating function of the
spanning $2$-forests of $\Gamma_n$, where the leftmost and the
rightmost vertices belong to the same connected component, and the
upmost vertex belongs to the second connected component.
Similarly, by rotation, we define $R_n(a,b,c)$ (respectively
$L_n(a,b,c)$) for the spanning $2$-forests of $\Gamma_n$, where
the rightmost (respectively leftmost) vertex is not in the same
connected component containing the two other outmost vertices;
\item $Q_n(a,b,c)=$ weighted generating function of
the spanning $3$-forests of $\Gamma_n$, where the three outmost
vertices belong to three different connected components.
\end{itemize}

Observe that, in this model, we need to introduce three different
functions $U_n(a,b,c), R_n(a,b,c)$ and $L_n(a,b,c)$, since the
edge labelling is not invariant with respect to a rotation of
$\frac{2\pi}{3}$ as in the previous case. On the other hand it is
clear that, for each $n\geq 1$, one has $U_n(1,1,1)=
R_n(1,1,1)=L_n(1,1,1)$ and this common value is equal to
$S_{n-1}(1,1,1)$, where $S_n(a,b,c)$ is the generating function
introduced in Section \ref{rotational}. In what follows, we will
often omit the argument $(a,b,c)$ of the generating functions.

\begin{teo}\label{geometrichedirezionale}
For each $n\geq 1$, the weighted generating functions
$T_n(a,b,c)$, $U_n(a,b,c)$, $R_n(a,b,c)$, $L_n(a,b,c)$ and
$Q_n(a,b,c)$ satisfy the following equations:
\begin{eqnarray}\label{Tdir}
T_{n+1} =2T_{n}^2(U_{n}+R_{n}+L_{n})
\end{eqnarray}
\begin{eqnarray}\label{Udir}
U_{n+1} = T_{n}U_{n}(2R_{n}+2L_{n}+3U_{n})+T_{n}^2Q_{n}
\end{eqnarray}
\begin{eqnarray}\label{Rdir}
R_{n+1} = T_{n}R_{n}(2L_{n}+2U_{n}+3R_{n})+T_{n}^2Q_{n}
\end{eqnarray}
\begin{eqnarray}\label{Ldir}
L_{n+1} = T_{n}L_{n}(2R_{n}+2U_{n}+3L_{n})+T_{n}^2Q_{n}
\end{eqnarray}
\begin{eqnarray}\label{Qdir}
Q_{n+1}&=&4T_{n}Q_{n}(U_{n}+R_{n}+L_{n})\\
&+&2\left(U_{n}^2(R_{n}+L_{n})+R_{n}^2(L_{n}+U_{n})+L_{n}^2(R_{n}+U_{n})\right)\nonumber\\
&+& 2 U_{n}R_{n}L_{n},\nonumber
\end{eqnarray}
with initial conditions
$$
T_1(a,b,c)=ab+ac+bc \qquad U_1(a,b,c) = b\qquad R_1(a,b,c) =
a\qquad L_1(a,b,c) = c\qquad Q_1(a,b,c) = 1.
$$
\end{teo}

\begin{proof}
It is easy to check that the initial conditions hold. Then, the
proof of each recursive equation follows the same strategy as in
Theorem \ref{equazionimodellofacile}.
\end{proof}
\begin{os}\rm Observe that, by replacing $U_n,R_n$ and $L_n$ with $S_n$, one finds again the
equations given for the rotational-invariant model in Theorem
\ref{equazionimodellofacile}.
\end{os}
\vspace{0.1cm} \noindent In order to get explicit solutions of the
equations given in Theorem \ref{geometrichedirezionale}, we put
$$
\phi_1(a,b,c)=ab+ac+bc \qquad \phi_2(a,b,c)=a+b+c
$$
$$
f(a,b,c)=3a^2b+3ab^2+3a^2c+3ac^2+3b^2c+3bc^2+7abc
$$
and let us define the function
$F:\mathbb{R}^3\longrightarrow\mathbb{R}^3$ as
$F(x,y,z)=(F_1(x,y,z),F_2(x,y,z),F_3(x,y,z))$, where
$$
F_1(x,y,z) = 3x^2+3xz+3xy+yz \qquad F_2(x,y,z) = 3y^2+3xy+3yz+xz
\qquad F_3(x,y,z)=3z^2+3xz+3yz+xy.
$$
Moreover, we denote by $F^{(k)}_i(a,b,c)$ the $i$-th coordinate of
the vector $F^{(k)}=F(\ldots F(F(a,b,c)))$, where the function $F$
is iterated $k$ times. Note that $F^{(1)}_i(a,b,c)=F_i(a,b,c)$,
for each $i=1,2,3$.

\noindent Finally, for each $k\geq 3$, put $\phi_k(a,b,c) =
\phi_{k-1}(F_1(a,b,c),F_2(a,b,c),F_3(a,b,c))$, so that
$$
\phi_k(a,b,c)=\phi_2\left(F^{(k-2)}(a,b,c)\right).
$$

\begin{teo}\label{noname}
The weighted generating functions $T_n(a,b,c)$, $U_n(a,b,c)$,
$R_n(a,b,c)$, $L_n(a,b,c)$ and $Q_n(a,b,c)$ satisfying Equations
\eqref{Tdir}, \eqref{Udir}, \eqref{Rdir}, \eqref{Ldir} and
\eqref{Qdir}, with the initial conditions given in Theorem
\ref{geometrichedirezionale}, are:
$$
T_n(a,b,c)=2^{\frac{3^n+6n-9}{12}}\prod_{k=1}^n\phi_k^{\frac{3^{n-k+1}+3}{6}}(a,b,c),
\qquad \mbox{for each }n\geq 1;
$$
$$
U_n(a,b,c)=2^{\frac{3^n-6n+3}{12}}\prod_{k=1}^{n-1}\phi_k^{\frac{3^{n-k+1}-3}{6}}(a,b,c)F_2^{(n-1)}(a,b,c),
\qquad \mbox{for each }n\geq 2;
$$
$$
R_n(a,b,c)=2^{\frac{3^n-6n+3}{12}}\prod_{k=1}^{n-1}\phi_k^{\frac{3^{n-k+1}-3}{6}}(a,b,c)F_1^{(n-1)}(a,b,c),
\qquad \mbox{for each }n\geq 2;
$$
$$
L_n(a,b,c)=2^{\frac{3^n-6n+3}{12}}\prod_{k=1}^{n-1}\phi_k^{\frac{3^{n-k+1}-3}{6}}(a,b,c)F_3^{(n-1)}(a,b,c),
\qquad \mbox{for each }n\geq 2;
$$
$$
Q_n(a,b,c)=2^{\frac{3^n-18n+39}{12}}\prod_{k=1}^{n-2}\phi_k^{\frac{3^{n-k+1}-9}{6}}(a,b,c)f\left(F^{(n-2)}(a,b,c)\right)
\qquad \mbox{for each }n\geq 3,
$$
with $U_1(a,b,c)=b, R_1(a,b,c)=a,  L_1(a,b,c)=c$, $Q_1(a,b,c)=1$
and $Q_2(a,b,c)=2f(a,b,c)$.
\end{teo}
\begin{proof}
The proof works by induction on $n$. One can directly find:
$$
U_2(a,b,c)=\phi_1(a,b,c)F_2(a,b,c) \qquad
R_2(a,b,c)=\phi_1(a,b,c)F_1(a,b,c) \qquad
L_2(a,b,c)=\phi_1(a,b,c)F_3(a,b,c)
$$
$$
T_1(a,b,c)=\phi_1(a,b,c) \qquad Q_3(a,b,c) = 2\phi_1^3(a,b,c)
f(F_1(a,b,c),F_2(a,b,c),F_3(a,b,c)),
$$
and so the basis of induction holds. We only prove the assertion
for $T_n(a,b,c)$, by showing that Equation \eqref{Tdir} is
satisfied (the computations in the other cases are similar but
more complicated). One has:
\begin{eqnarray*}
2T_{n}^2(U_{n}+R_{n}+L_{n}) &=& 2\cdot
2^{\frac{3^{n}+6n-9}{6}}\prod_{k=1}^{n}\phi_k^{\frac{3^{n-k+1}+3}{3}}(a,b,c)\cdot
2^{\frac{3^{n}-6n+3}{12}}\prod_{k=1}^{n-1}
\phi_k^{\frac{3^{n-k+1}-3}{6}}(a,b,c)\\
&\cdot&
\left(F_2^{(n-1)}(a,b,c)+F_1^{(n-1)}(a,b,c)+F_3^{(n-1)}(a,b,c)\right)\\
&=&
2^{\frac{3^{n+1}+6n-3}{12}}\prod_{k=1}^{n}\phi_k^{\frac{3^{n-k+2}+3}{6}}(a,b,c)\\&\cdot&\left(F_2^{(n-1)}(a,b,c)+F_1^{(n-1)}(a,b,c)+F_3^{(n-1)}(a,b,c)\right)\\
&=&
2^{\frac{3^{n+1}+6n-3}{12}}\prod_{k=1}^{n}\phi_k^{\frac{3^{n-k+2}+3}{6}}(a,b,c)\phi_2\!\left(F^{(n-1)}(a,b,c)\right)\\
&=&
2^{\frac{3^{n+1}+6n-3}{12}}\prod_{k=1}^{n+1}\phi_k^{\frac{3^{n-k+2}+3}{6}}(a,b,c)=T_{n+1}.
\end{eqnarray*}
\end{proof}

\subsection{Third model: the \lq\lq Schreier\rq\rq labelling}

Consider the graph $\Gamma_1$ in the picture below and define by
recurrence, for each $n\geq 1$, the graph $\Gamma_{n+1}$ as
constituted by the union of three copies of $\Gamma_{n}$ in the
following way: for each one of the outmost vertices of
$\Gamma_{n+1}$, the corresponding copy is given by the graph
$\Gamma_{n}$, reflected with respect to the bisectrix of the
corresponding angle.\unitlength=0.2mm
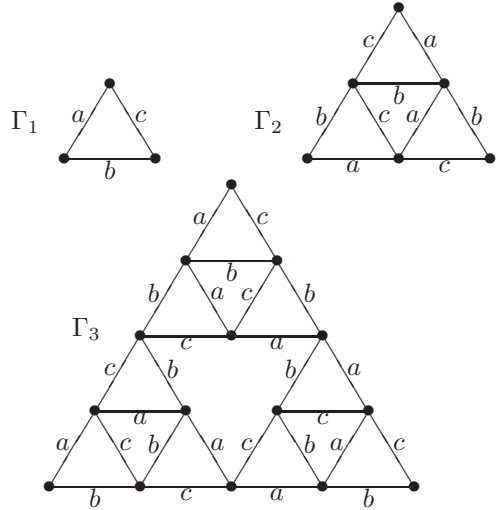
\begin{center}
\begin{picture}(400,105)

\put(55,30){$\Gamma_1$}\put(215,30){$\Gamma_2$}

\letvertex A=(120,60)\letvertex B=(90,10)\letvertex C=(150,10)

\letvertex D=(310,110)\letvertex E=(280,60)\letvertex F=(250,10)\letvertex G=(310,10)
\letvertex H=(370,10)\letvertex I=(340,60)

\drawvertex(A){$\bullet$}\drawvertex(B){$\bullet$}
\drawvertex(C){$\bullet$}\drawvertex(D){$\bullet$}
\drawvertex(E){$\bullet$}\drawvertex(F){$\bullet$}
\drawvertex(G){$\bullet$}\drawvertex(H){$\bullet$}
\drawvertex(I){$\bullet$}

\drawundirectededge(B,A){$a$} \drawundirectededge(C,B){$b$}
\drawundirectededge(A,C){$c$} \drawundirectededge(E,D){$c$}
\drawundirectededge(F,E){$b$} \drawundirectededge(G,F){$a$}

\drawundirectededge(H,G){$c$} \drawundirectededge(I,H){$b$}
\drawundirectededge(D,I){$a$} \drawundirectededge(I,E){$b$}
\drawundirectededge(E,G){$c$} \drawundirectededge(G,I){$a$}
\end{picture}
\end{center}
\begin{center}
\begin{picture}(400,200)

\put(95,110){$\Gamma_3$}

\letvertex A=(200,210)\letvertex B=(170,160)\letvertex C=(140,110)
\letvertex D=(110,60)\letvertex E=(80,10)\letvertex F=(140,10)\letvertex G=(200,10)
\letvertex H=(260,10)\letvertex I=(320,10)
\letvertex L=(290,60)\letvertex M=(260,110)\letvertex N=(230,160)
\letvertex O=(200,110)\letvertex P=(170,60)\letvertex Q=(230,60)

\drawvertex(A){$\bullet$}\drawvertex(B){$\bullet$}
\drawvertex(C){$\bullet$}\drawvertex(D){$\bullet$}
\drawvertex(E){$\bullet$}\drawvertex(F){$\bullet$}
\drawvertex(G){$\bullet$}\drawvertex(H){$\bullet$}
\drawvertex(I){$\bullet$}\drawvertex(L){$\bullet$}\drawvertex(M){$\bullet$}
\drawvertex(N){$\bullet$}\drawvertex(O){$\bullet$}
\drawvertex(P){$\bullet$}\drawvertex(Q){$\bullet$}

\drawundirectededge(E,D){$a$} \drawundirectededge(D,C){$c$}
\drawundirectededge(C,B){$b$} \drawundirectededge(B,A){$a$}
\drawundirectededge(A,N){$c$} \drawundirectededge(N,M){$b$}
\drawundirectededge(M,L){$a$} \drawundirectededge(L,I){$c$}
\drawundirectededge(I,H){$b$} \drawundirectededge(H,G){$a$}
\drawundirectededge(G,F){$c$} \drawundirectededge(F,E){$b$}
\drawundirectededge(N,B){$b$} \drawundirectededge(O,C){$c$}
\drawundirectededge(M,O){$a$} \drawundirectededge(P,D){$a$}
\drawundirectededge(L,Q){$c$} \drawundirectededge(B,O){$a$}
\drawundirectededge(O,N){$c$} \drawundirectededge(C,P){$b$}
\drawundirectededge(P,G){$a$} \drawundirectededge(D,F){$c$}
\drawundirectededge(Q,M){$b$} \drawundirectededge(G,Q){$c$}
\drawundirectededge(H,L){$a$}\drawundirectededge(F,P){$b$}
\drawundirectededge(Q,H){$b$}
\end{picture}
\end{center}
\begin{os}\rm
For each $n\geq 1$, the graph $\Gamma_n$ coincides with the graph
obtained from the Schreier graph $\Sigma_n$ of the Hanoi Towers
Group $H^{(3)}$ by deleting the three loops and by contracting all
the edges joining two different elementary triangles, keeping the
labels in $\Sigma_n$ on the remaining edges. (See Section
\ref{SchreierGraphs}.)
\end{os}

Define the generating functions $T_n(a,b,c)$, $U_n(a,b,c)$,
$R_n(a,b,c)$, $L_n(a,b,c)$ and $Q_n(a,b,c)$ to have the same
meaning as in the case of the directional labelling (Section
\ref{directional}). In what follows, we will often omit the
argument $(a,b,c)$ of the generating functions.

\begin{teo}\label{equazionisoddisfatte}
For each $n\geq 1$, the weighted generating functions
$T_n(a,b,c)$, $U_n(a,b,c)$, $R_n(a,b,c)$, $L_n(a,b,c)$ and
$Q_n(a,b,c)$ satisfy the following equations:
\begin{eqnarray}\label{1vec}
T_{n+1} =2T_{n}^2\left(U_{n}+R_{n}+L_{n}\right)
\end{eqnarray}
\begin{eqnarray}\label{2vec}
U_{n+1} =
T_{n}\left(3L_{n}R_{n}+U_{n}R_{n}+U_{n}L_{n}+2U_{n}^2\right)+T_{n}^2Q_{n}
\end{eqnarray}
\begin{eqnarray}\label{4vec}
R_{n+1} =
T_{n}\left(3U_{n}L_{n}+U_{n}R_{n}+R_{n}L_{n}+2R_{n}^2\right)+T_{n}^2Q_{n}
\end{eqnarray}
\begin{eqnarray}\label{3vec}
L_{n+1} =
T_{n}\left(3U_{n}R_{n}+L_{n}U_{n}+R_{n}L_{n}+2L_{n}^2\right)+T_{n}^2Q_{n}
\end{eqnarray}
\begin{eqnarray}\label{5vec}
Q_{n+1} &=&4T_{n}Q_{n}\left(U_{n}+R_{n}+L_{n}\right)\\
&+&2\left(U_{n}^2(L_{n}+R_{n})+R_{n}^2(U_{n}+L_{n})+L_{n}^2(R_{n}+U_{n})\right)
\nonumber\\
&+&2U_{n}R_{n}L_{n},\nonumber
\end{eqnarray}
with initial conditions
$$
T_1(a,b,c)=ab+ac+bc\qquad U_1(a,b,c) = b\qquad R_1(a,b,c) =
a\qquad L_1(a,b,c) = c\qquad Q_1(a,b,c) = 1.
$$
\end{teo}

\begin{proof}
It is easy to check that the initial conditions hold. Then, the
proof of each recursive equation follows the same strategy as in
Theorem \ref{equazionimodellofacile}.
\end{proof}
\begin{os}\rm Observe that, by replacing $U_n$, $R_n$ and $L_n$ with $S_n$, one finds again the
equations obtained for the rotational-invariant model in Theorem
\ref{equazionimodellofacile}.
\end{os}

Put
$$
\psi_1(a,b,c)=ab+ac+bc  \qquad \psi_2(a,b,c)=a+b+c
$$
$$
g(a,b,c)=3a^2b+3ab^2+3a^2c+3ac^2+3b^2c+3bc^2+7abc
$$
and let us define the function
$G:\mathbb{R}^3\longrightarrow\mathbb{R}^3$ as
$G(x,y,z)=(G_1(x,y,z),G_2(x,y,z),G_3(x,y,z))$, where
$$
G_1(x,y,z)=x^2+2yz+xy+xz \qquad G_2(x,y,z)=y^2+2xz+xy+yz \qquad
G_3(x,y,z)= z^2+2xy+xz+yz.
$$
Finally, for each $k\geq 3$, put $\psi_k(a,b,c) =
\psi_{k-1}(G_1(a,b,c),G_2(a,b,c),G_3(a,b,c))$, so that
$$
\psi_k(a,b,c)=\psi_2\left(G^{(k-2)}(a,b,c)\right).
$$
\begin{teo}\label{noname2}
The weighted generating functions $T_n(a,b,c)$, $U_n(a,b,c)$,
$R_n(a,b,c)$, $L_n(a,b,c)$ and $Q_n(a,b,c)$ satisfying Equations
\eqref{1vec}, \eqref{2vec}, \eqref{4vec}, \eqref{3vec} and
\eqref{5vec}, with the initial conditions given in Theorem
\ref{equazionisoddisfatte}, are:
$$
T_n(a,b,c)=2^{\frac{3^{n-1}-1}{2}}\prod_{k=1}^n\psi_k^{\frac{3^{n-k}+1}{2}}(a,b,c)
\qquad \mbox{for each }n\geq 1;
$$
$$
U_n(a,b,c)=2^{\frac{3^{n-1}-1}{2}}\prod_{k=1}^{n-1}
\psi_k^{\frac{3^{n-k}-1}{2}}(a,b,c)G_2^{(n-1)}(a,b,c) \qquad
\mbox{for each }n\geq 2;
$$
$$
R_n(a,b,c)=2^{\frac{3^{n-1}-1}{2}}\prod_{k=1}^{n-1}
\psi_k^{\frac{3^{n-k}-1}{2}}(a,b,c)G_1^{(n-1)}(a,b,c) \qquad
\mbox{for each }n\geq 2;
$$
$$
L_n(a,b,c)=2^{\frac{3^{n-1}-1}{2}}\prod_{k=1}^{n-1}
\psi_k^{\frac{3^{n-k}-1}{2}}(a,b,c)G_3^{(n-1)}(a,b,c) \qquad
\mbox{for each }n\geq 2;
$$
$$
Q_n(a,b,c)=2^{\frac{3^{n-1}-1}{2}}\prod_{k=1}^{n-2}\psi_k^{\frac{3^{n-k}-3}{2}}(a,b,c)g\left(G^{(n-2)}(a,b,c)\right)
\qquad \mbox{for each }n\geq 3,
$$
with $U_1(a,b,c) = b, R_1(a,b,c) = a, L_1(a,b,c) = c$, $Q_1(a,b,c)
= 1$ and $Q_2(a,b,c) = 2g(a,b,c)$.
\end{teo}

\begin{proof}
The proof is by induction on $n$. One can directly find:
$$
U_2(a,b,c)=2\psi_1(a,b,c)G_2(a,b,c) \quad
R_2(a,b,c)=2\psi_1(a,b,c)G_1(a,b,c) \quad
L_2(a,b,c)=2\psi_1(a,b,c)G_3(a,b,c)
$$
$$
T_1(a,b,c)=\psi_1(a,b,c)\qquad Q_3(a,b,c) = 2^4\psi_1^3(a,b,c)
g(G_1(a,b,c),G_2(a,b,c),G_3(a,b,c)),
$$
and so the basis on the induction holds. We only prove the
assertion for $T_n(a,b,c)$, by showing that Equation \eqref{1vec}
is satisfied (the computations in the other cases are similar but
more complicated). One has:
\begin{eqnarray*}
2T_{n}^2(U_{n}+R_{n}+L_{n}) &=&2\cdot
2^{3^{n-1}-1}\prod_{k=1}^{n}\psi_k^{3^{n-k}+1}(a,b,c)\cdot
2^{\frac{3^{n-1}-1}{2}}\prod_{k=1}^{n-1}
\psi_k^{\frac{3^{n-k}-1}{2}}(a,b,c)\\
&\cdot& (G_2^{(n-1)}(a,b,c)+G_1^{(n-1)}(a,b,c)+G_3^{(n-1)}(a,b,c))\\
&=&2^{\frac{3^{n}-1}{2}}\prod_{k=1}^{n}\psi_k^{\frac{3^{n-k+1}+1}{2}}(a,b,c)\psi_2(G^{(n-1)}(a,b,c))\\
&=&2^{\frac{3^{n}-1}{2}}\prod_{k=1}^{n+1}\psi_k^{\frac{3^{n-k+1}+1}{2}}(a,b,c)=T_{n+1}.
\end{eqnarray*}
\end{proof}
\begin{os}\rm
Note that the function $g(a,b,c)$ coincides with the function
$f(a,b,c)$, introduced in Section \ref{directional}. However, the
functions $G(x,y,z)$ and $F(x,y,z)$ do not coincide, which implies
that the functions $\psi_i(a,b,c)$ and $\phi_i(a,b,c)$,
factorizing the weighted generating functions of the spanning
trees in the directional model and in the Schreier model, are
different.
\end{os}


\section{Spanning trees on the Schreier graphs of the Hanoi Towers
group}\label{sectionhanoi}

In this section, we study spanning trees on the Schreier graphs of
the Hanoi Towers group $H^{(3)}$, which are very similar to the
Sierpi\'{n}ski graphs studied in the previous sections. We start
with a combinatorial recursive approach as in Section
\ref{sectionsierpinski}, but this leads us to some equations that
we can explicitly solve only in the unweighted case. Hence, in
order to compute the weighted generating function, we follow a
different approach: we use the self-similar presentation of the
generators of the group to write the adjacency matrix of the
graphs. Therefore, we are able to write the associated Laplace
matrix and then we get the weighted generating function of the
spanning trees by using a weighted version of Kirchhoff's
Matrix-Tree Theorem. The self-similarity of the group is the key
property that we use.

\subsection{The Schreier graphs of the Hanoi Towers group}\label{SchreierGraphs}

Recall that the Hanoi Towers group $H^{(3)}$ is generated by the
automorphisms of the ternary rooted tree having the following
self-similar form:
$$
a= (01)(id,id,a) \qquad  b=(02)(id,b,id) \qquad c=(12)(c,id,id),
$$
where $(01), (02)$ and $(12)$ are elements of the symmetric group
$Sym(3)$, acting on the set $X=\{0,1,2\}$. Observe that $a,b,c$
are involutions. The associated Schreier graphs are self-similar
in the sense of \cite{wagnerself}, that is, $\Sigma_{n+1}$
contains three copies of $\Sigma_{n}$  glued together by three
edges, that we call \lq\lq exceptional\rq\rq. By Definition
\ref{defischreiernovembre}, each vertex of $\Sigma_n$ corresponds
to a word of length $n$ in the alphabet $X$. The graphs
$\{\Sigma_n\}_{n\geq 1}$ can be recursively constructed via the
following substitutional rules \cite{hanoi}.\unitlength=0.3mm

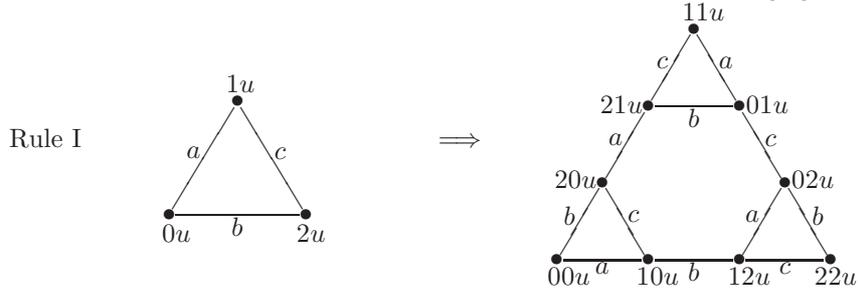
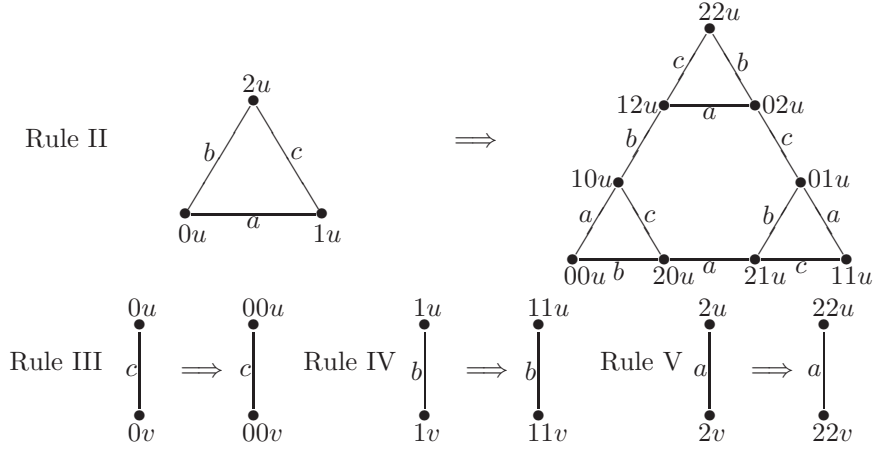
\begin{center}
\begin{picture}(400,115)
\letvertex A=(240,10)\letvertex B=(260,44)
\letvertex C=(280,78)\letvertex D=(300,112)
\letvertex E=(320,78)\letvertex F=(340,44)
\letvertex G=(360,10)\letvertex H=(320,10)\letvertex I=(280,10)

\letvertex L=(70,30)\letvertex M=(130,30)
\letvertex N=(100,80)

\put(236,-1){$00u$}\put(239,42){$20u$}\put(259,75){$21u$}
\put(295,116){$11u$}\put(323,75){$01u$}\put(343,42){$02u$}\put(353,-1){$22u$}
\put(315,-1){$12u$}\put(275,-1){$10u$}

\put(67,18){$0u$}\put(126,18){$2u$}\put(95,84){$1u$}\put(188,60){$\Longrightarrow$}
\put(0,60){Rule I}

\drawvertex(A){$\bullet$}\drawvertex(B){$\bullet$}
\drawvertex(C){$\bullet$}\drawvertex(D){$\bullet$}
\drawvertex(E){$\bullet$}\drawvertex(F){$\bullet$}
\drawvertex(G){$\bullet$}\drawvertex(H){$\bullet$}
\drawvertex(I){$\bullet$}
\drawundirectededge(A,B){$b$}\drawundirectededge(B,C){$a$}\drawundirectededge(C,D){$c$}
\drawundirectededge(D,E){$a$}\drawundirectededge(E,C){$b$}\drawundirectededge(E,F){$c$}\drawundirectededge(F,G){$b$}
\drawundirectededge(B,I){$c$}\drawundirectededge(H,F){$a$}\drawundirectededge(H,I){$b$}
\drawundirectededge(I,A){$a$}\drawundirectededge(G,H){$c$}

\drawvertex(L){$\bullet$}
\drawvertex(M){$\bullet$}\drawvertex(N){$\bullet$}
\drawundirectededge(M,L){$b$}\drawundirectededge(N,M){$c$}\drawundirectededge(L,N){$a$}

\end{picture}
\end{center}

\begin{center}
\begin{picture}(400,120)
\letvertex A=(240,10)\letvertex B=(260,44)
\letvertex C=(280,78)\letvertex D=(300,112)
\letvertex E=(320,78)\letvertex F=(340,44)
\letvertex G=(360,10)\letvertex H=(320,10)\letvertex I=(280,10)

\letvertex L=(70,30)\letvertex M=(130,30)
\letvertex N=(100,80)

\put(236,-1){$00u$}\put(239,42){$10u$}\put(259,75){$12u$}
\put(295,116){$22u$}\put(323,75){$02u$}\put(343,42){$01u$}\put(353,-1){$11u$}
\put(315,-1){$21u$}\put(275,-1){$20u$}

\put(67,18){$0u$}\put(126,18){$1u$}\put(95,84){$2u$}\put(188,60){$\Longrightarrow$}
\put(0,60){Rule II}
\drawvertex(A){$\bullet$}\drawvertex(B){$\bullet$}
\drawvertex(C){$\bullet$}\drawvertex(D){$\bullet$}
\drawvertex(E){$\bullet$}\drawvertex(F){$\bullet$}
\drawvertex(G){$\bullet$}\drawvertex(H){$\bullet$}
\drawvertex(I){$\bullet$}
\drawundirectededge(A,B){$a$}\drawundirectededge(B,C){$b$}\drawundirectededge(C,D){$c$}
\drawundirectededge(D,E){$b$}\drawundirectededge(E,C){$a$}\drawundirectededge(E,F){$c$}\drawundirectededge(F,G){$a$}
\drawundirectededge(B,I){$c$}\drawundirectededge(H,F){$b$}\drawundirectededge(H,I){$a$}
\drawundirectededge(I,A){$b$}\drawundirectededge(G,H){$c$}

\drawvertex(L){$\bullet$}
\drawvertex(M){$\bullet$}\drawvertex(N){$\bullet$}
\drawundirectededge(M,L){$a$}\drawundirectededge(N,M){$c$}\drawundirectededge(L,N){$b$}
\end{picture}
\end{center}
\begin{center}
\begin{picture}(400,60)
\letvertex A=(50,10)\letvertex B=(100,10)
\letvertex C=(175,10)\letvertex D=(225,10)
\letvertex E=(300,10)\letvertex F=(350,10)
\letvertex G=(50,50)\letvertex H=(100,50)
\letvertex I=(175,50)\letvertex L=(225,50)
\letvertex M=(300,50)\letvertex N=(350,50)

\put(45,54){$0u$}\put(45,-1){$0v$}
\put(95,-1){$00v$}\put(95,54){$00u$}\put(170,-1){$1v$}\put(170,54){$1u$}\put(220,-1){$11v$}
\put(220,54){$11u$}\put(295,-1){$2v$}\put(295,54){$2u$}\put(345,-1){$22v$}\put(345,54){$22u$}

\put(68,27){$\Longrightarrow$}\put(193,27){$\Longrightarrow$}\put(318,27){$\Longrightarrow$}
\put(-7,30){Rule III} \put(122,30){Rule IV} \put(252,30){Rule V}

\drawvertex(A){$\bullet$}\drawvertex(B){$\bullet$}
\drawvertex(C){$\bullet$}\drawvertex(D){$\bullet$}
\drawvertex(E){$\bullet$}\drawvertex(F){$\bullet$}
\drawvertex(G){$\bullet$}\drawvertex(H){$\bullet$}
\drawvertex(I){$\bullet$}\drawvertex(L){$\bullet$}
\drawvertex(M){$\bullet$}\drawvertex(N){$\bullet$}

\drawundirectededge(A,G){$c$}\drawundirectededge(B,H){$c$}\drawundirectededge(C,I){$b$}
\drawundirectededge(D,L){$b$}\drawundirectededge(E,M){$a$}\drawundirectededge(F,N){$a$}
\end{picture}
\end{center}
Notice that the word $u$ in Rules I and II can also be empty and
the words $u$ and $v$ in Rules III, IV, V can also satisfy $u=v$
(in this case we get the three loops of $\Sigma_n$). The starting
point of this recursive construction is the Schreier graph
$\Sigma_1$ of the first level. We also draw a picture of
$\Sigma_2$.
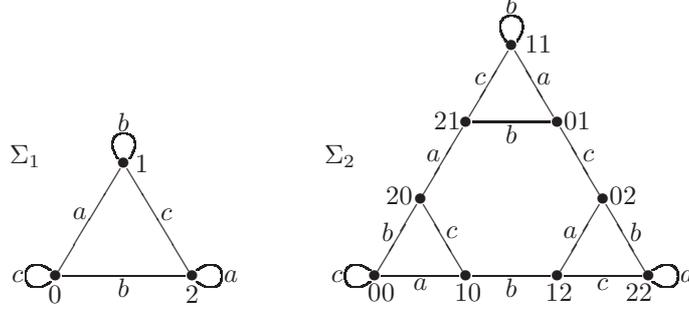
\begin{center}
\begin{picture}(400,125)

\letvertex L=(60,10)\letvertex M=(120,10)
\letvertex N=(90,60)

\put(57,-2){$0$}\put(117,-2){$2$}\put(95,56){$1$}\put(40,60){$\Sigma_1$}

\drawvertex(L){$\bullet$}
\drawvertex(M){$\bullet$}\drawvertex(N){$\bullet$}
\drawundirectededge(M,L){$b$}\drawundirectededge(N,M){$c$}\drawundirectededge(L,N){$a$}

\drawundirectedloop[r](M){$a$}\drawundirectedloop(N){$b$}\drawundirectedloop[l](L){$c$}

\letvertex A=(200,10)\letvertex B=(220,44)
\letvertex C=(240,78)

\letvertex D=(260,112)
\letvertex E=(280,78)\letvertex F=(300,44)
\letvertex G=(320,10)\letvertex H=(280,10)\letvertex I=(240,10)

\put(197,-1){$00$}\put(205,42){$20$}\put(226,75){$21$}
\put(266,109){$11$}\put(283,75){$01$}\put(303,42){$02$}\put(310,-1){$22$}
\put(275,-1){$12$}\put(235,-1){$10$}\put(178,60){$\Sigma_2$}

\drawvertex(A){$\bullet$}\drawvertex(B){$\bullet$}
\drawvertex(C){$\bullet$} \drawvertex(D){$\bullet$}
\drawvertex(E){$\bullet$}\drawvertex(F){$\bullet$}
\drawvertex(G){$\bullet$}\drawvertex(H){$\bullet$}
\drawvertex(I){$\bullet$}
\drawundirectededge(A,B){$b$}\drawundirectededge(B,C){$a$}\drawundirectededge(C,D){$c$}
\drawundirectededge(D,E){$a$}\drawundirectededge(E,C){$b$}\drawundirectededge(E,F){$c$}\drawundirectededge(F,G){$b$}
\drawundirectededge(B,I){$c$}\drawundirectededge(H,F){$a$}\drawundirectededge(H,I){$b$}
\drawundirectededge(I,A){$a$}\drawundirectededge(G,H){$c$}

\drawundirectedloop[l](A){$c$}\drawundirectedloop(D){$b$}\drawundirectedloop[r](G){$a$}

\end{picture}
\end{center}

For each $n\geq 1$, the graph $\Sigma_{n+1}$ can be also obtained
in the following recursive way: we take the union of three copies
of $\Sigma_{n}$ and, for each one of the outmost vertices of
$\Sigma_{n+1}$, the corresponding copy is reflected with respect
to the bisectrix of the corresponding angle.

\begin{os}\rm
Observe that, for each $n\geq 1$, the graph $\Sigma_n$ has three
loops, centered at the vertices $0^n,1^n$ and $2^n$, labelled by
$c,b$ and $a$, respectively. This is an easy consequence of the
definition of the generators $a,b$ and $c$ of $H^{(3)}$. Moreover,
these are the only loops in $\Sigma_n$. However, in what follows,
we will be studying spanning trees on the Schreier graphs
$\Sigma_n$ without loops, since a spanning tree of $\Sigma_n$
cannot contain any loop. By abuse of notation, we still denote by
$\Sigma_n$ the graph without loops.
\end{os}


\subsection{Computation of the complexity}\label{spannhanoi}

For each $n\geq 1$, let $T_n(a,b,c)$, $U_n(a,b,c)$, $R_n(a,b,c)$,
$L_n(a,b,c)$ and $Q_n(a,b,c)$ be the generating functions having
the same meaning as in Section \ref{sectionsierpinski}. In what
follows, we will often omit the argument $(a,b,c)$ in the
generating functions.

\begin{teo}\label{Teoremabrutti}
For each $n\geq 1$, the weighted generating functions
$T_n(a,b,c)$, $U_n(a,b,c)$, $R_n(a,b,c)$, $L_n(a,b,c)$ and
$Q_n(a,b,c)$ satisfy the following equations:
\begin{eqnarray}\label{Tbrutto}
T_{n+1} = T_{n}^3(ab+ac+bc)+2abcT_{n}^2(U_{n}+R_{n}+L_{n})
\end{eqnarray}
\begin{eqnarray}\label{Ubrutto}
U_{n+1} &=&
bT_{n}^3+T_{n}^2\left((ab+ac+bc)U_{n}+2b(aR_{n}+cL_{n})\right)\\
&+&
abcT_{n}\left(3R_{n}L_{n}+U_{n}(L_{n}+R_{n}+2U_{n})\right)+abcT_{n}^2Q_{n}\nonumber
\end{eqnarray}
\begin{eqnarray}\label{Rbrutto}
R_{n+1} &=&
aT_{n}^3+T_{n}^2((ab+ac+bc)R_{n}+2a(bU_{n}+cL_{n}))\\
&+&
abcT_{n}(3U_{n}L_{n}+R_{n}(L_{n}+U_{n}+2R_{n}))+abcT_{n}^2Q_{n}\nonumber
\end{eqnarray}
\begin{eqnarray}\label{Lbrutto}
L_{n+1} &=&
cT_{n}^3+T_{n}^2((ab+ac+bc)L_{n}+2c(aR_{n}+bU_{n}))\\
&+&
abcT_{n}(3R_{n}U_{n}+L_{n}(U_{n}+R_{n}+2L_{n}))+abcT_{n}^2Q_{n}\nonumber
\end{eqnarray}
\begin{eqnarray}\label{Qbrutto}
Q_{n+1}&=& 4abcT_{n}Q_{n}(U_{n}+R_{n}+L_{n})\\
&+& T_{n}^2 ((2b+a+c)U_{n}+(2a+b+c)R_{n}+(2c+a+b)L_{n})\nonumber \\
&+& T_{n}^2Q_{n}(ab+ac+bc) + T_{n}^3\nonumber\\
&+&
2abc\left(U^2_{n}(R_{n}+L_{n})+R_{n}^2(U_{n}+L_{n})+L_{n}^2(U_{n}+R_{n})+U_{n}R_{n}L_{n}\right)\nonumber\\
&+& 2T_{n}\left(U_{n}R_{n}(ac+bc+2ab) +
U_{n}L_{n}(ab+ac+2bc)+R_{n}L_{n}(ab+bc+2ac)\right.\nonumber\\
&+&\left.bU_{n}^2(a+c)+aR_{n}^2(b+c)+cL_{n}^2(a+b)\right)\nonumber,
\end{eqnarray}
with initial conditions
$$
T_1(a,b,c)=ab+ac+bc\qquad U_1(a,b,c) = b\qquad R_1(a,b,c) =
a\qquad L_1(a,b,c) = c\qquad Q_1(a,b,c) = 1.
$$
\end{teo}
\begin{proof}
The graph $\Sigma_{n+1}$ can be represented as
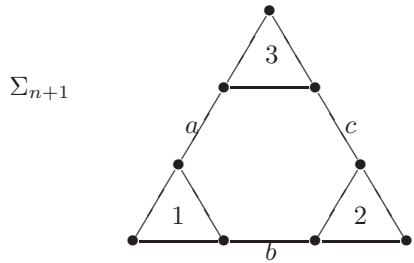
\begin{center}
\begin{picture}(400,112)
\letvertex A=(140,10)\letvertex B=(160,44)
\letvertex C=(180,78)\letvertex D=(200,112)
\letvertex E=(220,78)\letvertex F=(240,44)
\letvertex G=(260,10)\letvertex H=(220,10)\letvertex I=(180,10)

\put(86,75){$\Sigma_{n+1}$}

\put(157,18){1} \put(237,18){2} \put(198,89){3}

\put(163,58){$a$} \put(233,58){$c$} \put(198,2){$b$}

\drawvertex(A){$\bullet$}\drawvertex(B){$\bullet$}
\drawvertex(C){$\bullet$}\drawvertex(D){$\bullet$}
\drawvertex(E){$\bullet$}\drawvertex(F){$\bullet$}
\drawvertex(G){$\bullet$}\drawvertex(H){$\bullet$}
\drawvertex(I){$\bullet$}

\drawundirectededge(D,E){}
\drawundirectededge(B,C){}\drawundirectededge(H,F){}\drawundirectededge(I,A){}
\drawundirectededge(E,C){}\drawundirectededge(A,B){}
\drawundirectededge(H,I){}\drawundirectededge(F,G){}
\drawundirectededge(C,D){}
\drawundirectededge(E,F){}\drawundirectededge(B,I){}\drawundirectededge(G,H){}
\end{picture}
\end{center}
where the triangles 1, 2 and 3 represent subgraphs of
$\Sigma_{n+1}$ isomorphic to the graph $\Sigma_{n}$. The edges
labelled by $a,b$ and $c$ in the picture above are the exceptional
edges joining the different copies of $\Sigma_{n}$. In the
pictures of this proof we will use the same conventions as in
Theorem \ref{equazionimodellofacile}.\\
\indent Now, a spanning tree of $\Sigma_{n+1}$ can be obtained by
choosing a spanning tree for each one of the triangles 1, 2 and 3,
and omitting one of the edges $a,b$ and $c$ in order to have no
cycle. This gives the contribution $T_{n}^3(ab+ac+bc)$ to
$T_{n+1}$. A spanning tree of $\Sigma_{n+1}$ can also be obtained
by choosing a spanning tree for two of the triangles 1, 2 and 3,
and a 2-forest in the third triangle. This time, we do not need to
omit any one of edges $a,b$ and $c$ (see the picture
below).\unitlength=0,2mm
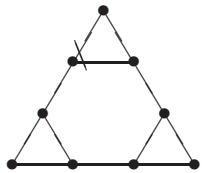
\begin{center}
\begin{picture}(400,125)
\letvertex A=(140,10)\letvertex B=(160,44)
\letvertex C=(180,78)\letvertex D=(200,112)
\letvertex E=(220,78)\letvertex F=(240,44)
\letvertex G=(260,10)\letvertex H=(220,10)\letvertex I=(180,10)
\letvertex z=(180,95)\letvertex y=(190,70)
\drawundirectededge(z,y){}

\drawvertex(A){$\bullet$}\drawvertex(B){$\bullet$}
\drawvertex(C){$\bullet$}\drawvertex(D){$\bullet$}
\drawvertex(E){$\bullet$}\drawvertex(F){$\bullet$}
\drawvertex(G){$\bullet$}\drawvertex(H){$\bullet$}
\drawvertex(I){$\bullet$}

\drawundirectededge(E,C){}\drawundirectededge(B,I){}\drawundirectededge(H,F){}
\drawundirectededge(A,B){}\drawundirectededge(B,C){}\drawundirectededge(C,D){}
\drawundirectededge(D,E){}\drawundirectededge(E,F){}
\drawundirectededge(F,G){}
\drawundirectededge(H,I){}\drawundirectededge(I,A){}\drawundirectededge(G,H){}
\end{picture}
\end{center}
This situation corresponds to the contribution
$2abcT_{n}^2(U_{n}+R_{n}+L_{n})$ to $T_{n+1}$ and Equation
(\ref{Tbrutto}) is proven.

We want to prove Equation (\ref{Ubrutto}). The two following
pictures show that a spanning $2$-forest of type $U_{n+1}$ can be
obtained starting from three spanning trees of level $n$ (in this
case we omit two exceptional edges), but also by taking two
spanning trees in two copies of $\Sigma_n$ and a spanning
$2$-forest of type $U_{n}$ in the third one (by omitting one of
the three exceptional edges).


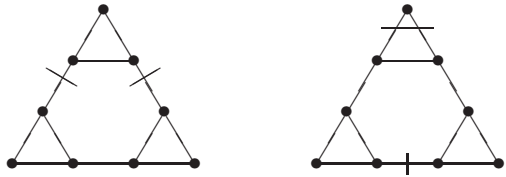
\begin{center}
\begin{picture}(400,110)
\letvertex A=(40,10)\letvertex B=(60,44)
\letvertex C=(80,78)\letvertex D=(100,112)
\letvertex E=(120,78)\letvertex F=(140,44)
\letvertex G=(160,10)\letvertex H=(120,10)\letvertex I=(80,10)
\letvertex z=(60,75)\letvertex y=(140,75)\letvertex
u=(85,60)\letvertex v=(115,60)

 \drawundirectededge(z,u){}\drawundirectededge(y,v){}

\drawvertex(A){$\bullet$}\drawvertex(B){$\bullet$}
\drawvertex(C){$\bullet$}\drawvertex(D){$\bullet$}
\drawvertex(E){$\bullet$}\drawvertex(F){$\bullet$}
\drawvertex(G){$\bullet$}\drawvertex(H){$\bullet$}
\drawvertex(I){$\bullet$}

\drawundirectededge(E,C){}\drawundirectededge(B,I){}\drawundirectededge(H,F){}
\drawundirectededge(A,B){}\drawundirectededge(B,C){}\drawundirectededge(C,D){}
\drawundirectededge(D,E){}\drawundirectededge(E,F){}
\drawundirectededge(F,G){}
\drawundirectededge(H,I){}\drawundirectededge(I,A){}\drawundirectededge(G,H){}

\letvertex a=(240,10)\letvertex b=(260,44)
\letvertex c=(280,78)\letvertex d=(300,112)
\letvertex e=(320,78)\letvertex f=(340,44)
\letvertex g=(360,10)\letvertex h=(320,10)\letvertex i=(280,10)

\letvertex x=(280,100)\letvertex w=(320,100)

\drawundirectededge(x,w){}
\letvertex j=(300,20)\letvertex J=(300,0)
\drawundirectededge(j,J){}

\drawvertex(a){$\bullet$}\drawvertex(b){$\bullet$}
\drawvertex(c){$\bullet$}\drawvertex(d){$\bullet$}
\drawvertex(e){$\bullet$}\drawvertex(f){$\bullet$}
\drawvertex(g){$\bullet$}\drawvertex(h){$\bullet$}
\drawvertex(i){$\bullet$}

\drawundirectededge(b,c){}\drawundirectededge(e,f){}
\drawundirectededge(h,i){}

 \drawundirectededge(a,b){}\drawundirectededge(c,d){}
\drawundirectededge(d,e){}\drawundirectededge(e,c){}
\drawundirectededge(f,g){}\drawundirectededge(b,i){}\drawundirectededge(h,f){}
\drawundirectededge(i,a){}\drawundirectededge(g,h){}
\end{picture}
\end{center}

More precisely, the first picture corresponds to a contribution
equal to $bT_{n}^3$, the second one to the contribution
$T_{n}^2U_{n}(ab+ac+bc)$ to $U_{n+1}$. Consider now the following
pictures.

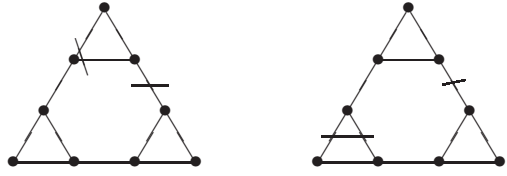
\begin{center}
\begin{picture}(400,110)
\letvertex A=(40,10)\letvertex B=(60,44)
\letvertex C=(80,78)\letvertex D=(100,112)
\letvertex E=(120,78)\letvertex F=(140,44)
\letvertex G=(160,10)\letvertex H=(120,10)\letvertex I=(80,10)

\letvertex z=(80,95)\letvertex Z=(90,65)
\letvertex y=(115,61)\letvertex Y=(145,61)
\drawundirectededge(z,Z){}\drawundirectededge(y,Y){}

\drawvertex(A){$\bullet$}\drawvertex(B){$\bullet$}
\drawvertex(C){$\bullet$}\drawvertex(D){$\bullet$}
\drawvertex(E){$\bullet$}\drawvertex(F){$\bullet$}
\drawvertex(G){$\bullet$}\drawvertex(H){$\bullet$}
\drawvertex(I){$\bullet$}

\drawundirectededge(E,C){}\drawundirectededge(B,I){}\drawundirectededge(H,F){}

\drawundirectededge(A,B){}\drawundirectededge(B,C){}\drawundirectededge(C,D){}
\drawundirectededge(D,E){}\drawundirectededge(E,F){}
\drawundirectededge(F,G){}
\drawundirectededge(H,I){}\drawundirectededge(I,A){}\drawundirectededge(G,H){}

\letvertex a=(240,10)\letvertex b=(260,44)
\letvertex c=(280,78)\letvertex d=(300,112)
\letvertex e=(320,78)\letvertex f=(340,44)
\letvertex g=(360,10)\letvertex h=(320,10)\letvertex i=(280,10)
\letvertex u=(240,27)\letvertex U=(280,27)
\letvertex v=(320,61)\letvertex V=(340,65)
\drawundirectededge(u,U){}\drawundirectededge(v,V){}

\drawvertex(a){$\bullet$}\drawvertex(b){$\bullet$}
\drawvertex(c){$\bullet$}\drawvertex(d){$\bullet$}
\drawvertex(e){$\bullet$}\drawvertex(f){$\bullet$}
\drawvertex(g){$\bullet$}\drawvertex(h){$\bullet$}
\drawvertex(i){$\bullet$}

\drawundirectededge(b,c){}\drawundirectededge(e,f){}
\drawundirectededge(h,i){}

\drawundirectededge(a,b){}\drawundirectededge(c,d){}
\drawundirectededge(d,e){}\drawundirectededge(e,c){}
\drawundirectededge(f,g){}\drawundirectededge(b,i){}\drawundirectededge(h,f){}
\drawundirectededge(i,a){}\drawundirectededge(g,h){}
\end{picture}
\end{center}
These configurations, together with their symmetric ones obtained
by reflecting with respect to the vertical axis, give a
contribution to $U_{n+1}$ equal to $2bT_{n}^2(aR_{n}+cL_{n})$.
Consider now the two following pictures.


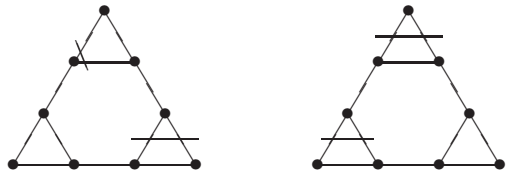
\begin{center}
\begin{picture}(400,110)
\letvertex A=(40,10)\letvertex B=(60,44)
\letvertex C=(80,78)\letvertex D=(100,112)
\letvertex E=(120,78)\letvertex F=(140,44)
\letvertex G=(160,10)\letvertex H=(120,10)\letvertex I=(80,10)

\letvertex z=(80,95)\letvertex y=(90,70)
\letvertex Z=(115,27)\letvertex Y=(165,27)
\drawundirectededge(z,y){}\drawundirectededge(Z,Y){}

\drawvertex(A){$\bullet$}\drawvertex(B){$\bullet$}
\drawvertex(C){$\bullet$}\drawvertex(D){$\bullet$}
\drawvertex(E){$\bullet$}\drawvertex(F){$\bullet$}
\drawvertex(G){$\bullet$}\drawvertex(H){$\bullet$}
\drawvertex(I){$\bullet$}

\drawundirectededge(E,C){}\drawundirectededge(B,I){}\drawundirectededge(H,F){}

\drawundirectededge(A,B){}\drawundirectededge(B,C){}\drawundirectededge(C,D){}
\drawundirectededge(D,E){}\drawundirectededge(E,F){}
\drawundirectededge(F,G){}
\drawundirectededge(H,I){}\drawundirectededge(I,A){}\drawundirectededge(G,H){}

\letvertex a=(240,10)\letvertex b=(260,44)
\letvertex c=(280,78)\letvertex d=(300,112)
\letvertex e=(320,78)\letvertex f=(340,44)
\letvertex g=(360,10)\letvertex h=(320,10)\letvertex i=(280,10)

\letvertex x=(275,95)\letvertex w=(325,95)
\drawundirectededge(x,w){}
\letvertex X=(240,27)\letvertex W=(280,27)
\drawundirectededge(X,W){}

\drawvertex(a){$\bullet$}\drawvertex(b){$\bullet$}
\drawvertex(c){$\bullet$}\drawvertex(d){$\bullet$}
\drawvertex(e){$\bullet$}\drawvertex(f){$\bullet$}
\drawvertex(g){$\bullet$}\drawvertex(h){$\bullet$}
\drawvertex(i){$\bullet$}

\drawundirectededge(b,c){}\drawundirectededge(e,f){}
\drawundirectededge(h,i){}

 \drawundirectededge(a,b){}\drawundirectededge(c,d){}
\drawundirectededge(d,e){}\drawundirectededge(e,c){}
\drawundirectededge(f,g){}\drawundirectededge(b,i){}\drawundirectededge(h,f){}
\drawundirectededge(i,a){}\drawundirectededge(g,h){}
\end{picture}
\end{center}
The left picture, together with its symmetric, gives the
contribution $2abcT_{n}R_{n}L_{n}$; the right one, together with
its symmetric, gives the contribution
$abcT_{n}U_{n}(R_{n}+L_{n})$. Consider now the two following
situations.


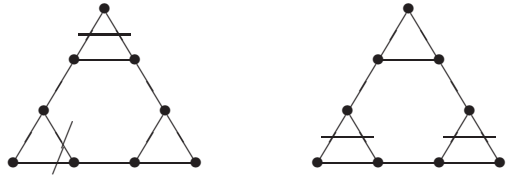
\begin{center}
\begin{picture}(400,110)
\letvertex A=(40,10)\letvertex B=(60,44)
\letvertex C=(80,78)\letvertex D=(100,112)
\letvertex E=(120,78)\letvertex F=(140,44)
\letvertex G=(160,10)\letvertex H=(120,10)\letvertex I=(80,10)

\letvertex z=(80,95)\letvertex y=(120,95)
\drawundirectededge(z,y){}
\letvertex Z=(80,40)\letvertex Y=(65,0)
\drawundirectededge(Z,Y){}

\drawvertex(A){$\bullet$}\drawvertex(B){$\bullet$}
\drawvertex(C){$\bullet$}\drawvertex(D){$\bullet$}
\drawvertex(E){$\bullet$}\drawvertex(F){$\bullet$}
\drawvertex(G){$\bullet$}\drawvertex(H){$\bullet$}
\drawvertex(I){$\bullet$}

\drawundirectededge(E,C){}\drawundirectededge(B,I){}\drawundirectededge(H,F){}

\drawundirectededge(A,B){}\drawundirectededge(B,C){}\drawundirectededge(C,D){}
\drawundirectededge(D,E){}\drawundirectededge(E,F){}
\drawundirectededge(F,G){}
\drawundirectededge(H,I){}\drawundirectededge(I,A){}\drawundirectededge(G,H){}

\letvertex a=(240,10)\letvertex b=(260,44)
\letvertex c=(280,78)\letvertex d=(300,112)
\letvertex e=(320,78)\letvertex f=(340,44)
\letvertex g=(360,10)\letvertex h=(320,10)\letvertex i=(280,10)

\letvertex x=(240,27)\letvertex w=(360,27)\letvertex X=(280,27)\letvertex W=(320,27)

\drawundirectededge(x,X){}\drawundirectededge(w,W){}
\drawvertex(a){$\bullet$}\drawvertex(b){$\bullet$}
\drawvertex(c){$\bullet$}\drawvertex(d){$\bullet$}
\drawvertex(e){$\bullet$}\drawvertex(f){$\bullet$}
\drawvertex(g){$\bullet$}\drawvertex(h){$\bullet$}
\drawvertex(i){$\bullet$}

\drawundirectededge(b,c){}\drawundirectededge(e,f){}
\drawundirectededge(h,i){}

\drawundirectededge(a,b){}\drawundirectededge(c,d){}
\drawundirectededge(d,e){}\drawundirectededge(e,c){}
\drawundirectededge(f,g){}\drawundirectededge(b,i){}\drawundirectededge(h,f){}
\drawundirectededge(i,a){}\drawundirectededge(g,h){}
\end{picture}
\end{center}
The picture on the left, together with its symmetric, gives a
contribution equal to $2abcT_{n}U_{n}^2$ to $U_{n+1}$. The picture
on the right contributes by the summand $abcT_{n}R_{n}L_{n}$.
Finally, the contribution $abcT_{n}^2Q_{n}$ is described by the
following picture.


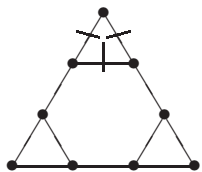
\begin{center}
\begin{picture}(400,110)
\letvertex A=(140,10)\letvertex B=(160,44)
\letvertex C=(180,78)\letvertex D=(200,112)
\letvertex E=(220,78)\letvertex F=(240,44)
\letvertex G=(260,10)\letvertex H=(220,10)\letvertex I=(180,10)

\letvertex z=(200,95)\letvertex y=(180,100)
\drawundirectededge(z,y){}\letvertex Z=(220,100)\letvertex
Y=(200,70) \drawundirectededge(z,Z){}\drawundirectededge(z,Y){}

\drawvertex(A){$\bullet$}\drawvertex(B){$\bullet$}
\drawvertex(C){$\bullet$}\drawvertex(D){$\bullet$}
\drawvertex(E){$\bullet$}\drawvertex(F){$\bullet$}
\drawvertex(G){$\bullet$}\drawvertex(H){$\bullet$}
\drawvertex(I){$\bullet$}

\drawundirectededge(E,C){}\drawundirectededge(B,I){}\drawundirectededge(H,F){}

\drawundirectededge(A,B){}\drawundirectededge(B,C){}\drawundirectededge(C,D){}
\drawundirectededge(D,E){}\drawundirectededge(E,F){}
\drawundirectededge(F,G){}
\drawundirectededge(H,I){}\drawundirectededge(I,A){}\drawundirectededge(G,H){}

\end{picture}
\end{center}
So we have proven Equation (\ref{Ubrutto}) about $U_{n+1}$.
Equations \eqref{Rbrutto} and \eqref{Lbrutto} can be proven in a
similar way.

We want to prove now Equation (\ref{Qbrutto}) about $Q_{n+1}$.
Consider the following pictures.
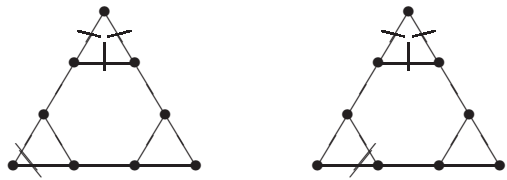
\begin{center}
\begin{picture}(400,110)
\letvertex A=(40,10)\letvertex B=(60,44)
\letvertex C=(80,78)\letvertex D=(100,112)
\letvertex E=(120,78)\letvertex F=(140,44)
\letvertex G=(160,10)\letvertex H=(120,10)\letvertex I=(80,10)

\letvertex z=(100,95)\letvertex y=(80,100)\letvertex Z=(120,100)
\letvertex Y=(100,70)\letvertex K=(40,27)\letvertex k=(60,0)

\drawundirectededge(z,y){}
\drawundirectededge(z,Z){}\drawundirectededge(z,Y){}
\drawundirectededge(k,K){}

\drawvertex(A){$\bullet$}\drawvertex(B){$\bullet$}
\drawvertex(C){$\bullet$}\drawvertex(D){$\bullet$}
\drawvertex(E){$\bullet$}\drawvertex(F){$\bullet$}
\drawvertex(G){$\bullet$}\drawvertex(H){$\bullet$}
\drawvertex(I){$\bullet$}

\drawundirectededge(E,C){}\drawundirectededge(B,I){}\drawundirectededge(H,F){}

\drawundirectededge(A,B){}\drawundirectededge(B,C){}\drawundirectededge(C,D){}
\drawundirectededge(D,E){}\drawundirectededge(E,F){}
\drawundirectededge(F,G){}
\drawundirectededge(H,I){}\drawundirectededge(I,A){}\drawundirectededge(G,H){}

\letvertex a=(240,10)\letvertex b=(260,44)
\letvertex c=(280,78)\letvertex d=(300,112)
\letvertex e=(320,78)\letvertex f=(340,44)
\letvertex g=(360,10)\letvertex h=(320,10)\letvertex i=(280,10)
\letvertex z=(300,95)\letvertex y=(280,100)\letvertex Z=(320,100)
\letvertex Y=(300,70)\letvertex K=(280,27)\letvertex k=(260,0)

\drawundirectededge(z,y){}
\drawundirectededge(z,Z){}\drawundirectededge(z,Y){}
\drawundirectededge(k,K){}

\drawvertex(a){$\bullet$}\drawvertex(b){$\bullet$}
\drawvertex(c){$\bullet$}\drawvertex(d){$\bullet$}
\drawvertex(e){$\bullet$}\drawvertex(f){$\bullet$}
\drawvertex(g){$\bullet$}\drawvertex(h){$\bullet$}
\drawvertex(i){$\bullet$}

\drawundirectededge(b,c){}\drawundirectededge(e,f){}
\drawundirectededge(h,i){}

\drawundirectededge(a,b){}\drawundirectededge(c,d){}
\drawundirectededge(d,e){}\drawundirectededge(e,c){}
\drawundirectededge(f,g){}\drawundirectededge(b,i){}\drawundirectededge(h,f){}
\drawundirectededge(i,a){}\drawundirectededge(g,h){}
\end{picture}
\end{center}

They provide, by symmetry, a contribution equal to
$4abcT_{n}Q_{n}(U_{n}+R_{n}+L_{n})$. Next, consider the following
cases.
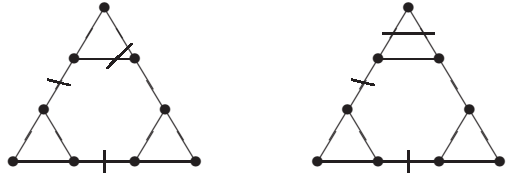
\begin{center}
\begin{picture}(400,110)
\letvertex A=(40,10)\letvertex B=(60,44)
\letvertex C=(80,78)\letvertex D=(100,112)
\letvertex E=(120,78)\letvertex F=(140,44)
\letvertex G=(160,10)\letvertex H=(120,10)\letvertex I=(80,10)

\letvertex z=(60,65)\letvertex y=(80,60)\letvertex Z=(100,70)
\letvertex Y=(120,90)\letvertex K=(100,20)\letvertex k=(100,0)

\drawundirectededge(z,y){} \drawundirectededge(Y,Z){}
\drawundirectededge(k,K){}

\drawvertex(A){$\bullet$}\drawvertex(B){$\bullet$}
\drawvertex(C){$\bullet$}\drawvertex(D){$\bullet$}
\drawvertex(E){$\bullet$}\drawvertex(F){$\bullet$}
\drawvertex(G){$\bullet$}\drawvertex(H){$\bullet$}
\drawvertex(I){$\bullet$}

\drawundirectededge(E,C){}\drawundirectededge(B,I){}\drawundirectededge(H,F){}

\drawundirectededge(A,B){}\drawundirectededge(B,C){}\drawundirectededge(C,D){}
\drawundirectededge(D,E){}\drawundirectededge(E,F){}
\drawundirectededge(F,G){}
\drawundirectededge(H,I){}\drawundirectededge(I,A){}\drawundirectededge(G,H){}

\letvertex a=(240,10)\letvertex b=(260,44)
\letvertex c=(280,78)\letvertex d=(300,112)
\letvertex e=(320,78)\letvertex f=(340,44)
\letvertex g=(360,10)\letvertex h=(320,10)\letvertex i=(280,10)

\letvertex u=(260,65)\letvertex v=(280,60)\letvertex U=(280,95)
\letvertex V=(320,95)\letvertex n=(300,20)\letvertex N=(300,0)

\drawundirectededge(u,v){} \drawundirectededge(U,V){}
\drawundirectededge(n,N){}

\drawvertex(a){$\bullet$}\drawvertex(b){$\bullet$}
\drawvertex(c){$\bullet$}\drawvertex(d){$\bullet$}
\drawvertex(e){$\bullet$}\drawvertex(f){$\bullet$}
\drawvertex(g){$\bullet$}\drawvertex(h){$\bullet$}
\drawvertex(i){$\bullet$}

\drawundirectededge(b,c){}\drawundirectededge(e,f){}
\drawundirectededge(h,i){}

 \drawundirectededge(a,b){}\drawundirectededge(c,d){}
\drawundirectededge(d,e){}\drawundirectededge(e,c){}
\drawundirectededge(f,g){}\drawundirectededge(b,i){}\drawundirectededge(h,f){}
\drawundirectededge(i,a){}\drawundirectededge(g,h){}
\end{picture}
\end{center}
The picture on the left gives, by symmetry, a contribution equal
to $2T_{n}^2 (bU_{n}+aR_{n}+cL_{n})$. The picture on the right
gives $T_{n}^2 ((a+c)U_{n}+(b+c)R_{n}+(a+b)L_{n})$. Now consider
the following cases.
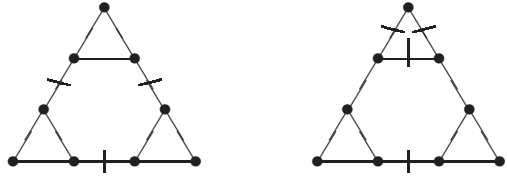
\begin{center}
\begin{picture}(400,110)
\letvertex A=(40,10)\letvertex B=(60,44)
\letvertex C=(80,78)\letvertex D=(100,112)
\letvertex E=(120,78)\letvertex F=(140,44)
\letvertex G=(160,10)\letvertex H=(120,10)\letvertex I=(80,10)

\letvertex z=(60,65)\letvertex y=(80,60)\letvertex Z=(140,65)
\letvertex Y=(120,60)\letvertex K=(100,20)\letvertex k=(100,0)

\drawundirectededge(z,y){} \drawundirectededge(Y,Z){}
\drawundirectededge(k,K){}

\drawvertex(A){$\bullet$}\drawvertex(B){$\bullet$}
\drawvertex(C){$\bullet$}\drawvertex(D){$\bullet$}
\drawvertex(E){$\bullet$}\drawvertex(F){$\bullet$}
\drawvertex(G){$\bullet$}\drawvertex(H){$\bullet$}
\drawvertex(I){$\bullet$}

\drawundirectededge(E,C){}\drawundirectededge(B,I){}\drawundirectededge(H,F){}

\drawundirectededge(A,B){}\drawundirectededge(B,C){}\drawundirectededge(C,D){}
\drawundirectededge(D,E){}\drawundirectededge(E,F){}
\drawundirectededge(F,G){}
\drawundirectededge(H,I){}\drawundirectededge(I,A){}\drawundirectededge(G,H){}

\letvertex a=(240,10)\letvertex b=(260,44)
\letvertex c=(280,78)\letvertex d=(300,112)
\letvertex e=(320,78)\letvertex f=(340,44)
\letvertex g=(360,10)\letvertex h=(320,10)\letvertex i=(280,10)

\letvertex u=(300,95)\letvertex v=(280,100)\letvertex U=(320,100)
\letvertex V=(300,70)\letvertex n=(300,20)\letvertex N=(300,0)

\drawundirectededge(u,v){}
\drawundirectededge(U,u){}\drawundirectededge(V,u){}
\drawundirectededge(n,N){}

\drawvertex(a){$\bullet$}\drawvertex(b){$\bullet$}
\drawvertex(c){$\bullet$}\drawvertex(d){$\bullet$}
\drawvertex(e){$\bullet$}\drawvertex(f){$\bullet$}
\drawvertex(g){$\bullet$}\drawvertex(h){$\bullet$}
\drawvertex(i){$\bullet$}

\drawundirectededge(b,c){}\drawundirectededge(e,f){}
\drawundirectededge(h,i){}

 \drawundirectededge(a,b){}\drawundirectededge(c,d){}
\drawundirectededge(d,e){}\drawundirectededge(e,c){}
\drawundirectededge(f,g){}\drawundirectededge(b,i){}\drawundirectededge(h,f){}
\drawundirectededge(i,a){}\drawundirectededge(g,h){}
\end{picture}
\end{center}

The picture on the left gives the contribution $T_{n}^3$ to
$Q_{n+1}$, the picture on the right contributes by
$T_{n}^2Q_{n}(ab+ac+bc)$. Consider the following pictures.
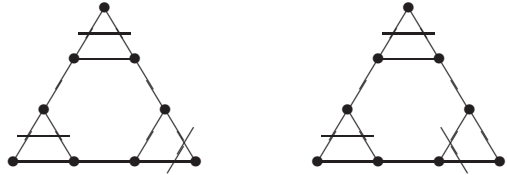
\begin{center}
\begin{picture}(400,110)
\letvertex A=(40,10)\letvertex B=(60,44)
\letvertex C=(80,78)\letvertex D=(100,112)
\letvertex E=(120,78)\letvertex F=(140,44)
\letvertex G=(160,10)\letvertex H=(120,10)\letvertex I=(80,10)

\letvertex z=(80,95)\letvertex y=(120,95)\drawundirectededge(z,y){}

\letvertex Z=(40,27)\letvertex Y=(80,27) \drawundirectededge(Y,Z){}

\letvertex u=(160,35)\letvertex v=(140,0) \drawundirectededge(u,v){}

\drawvertex(A){$\bullet$}\drawvertex(B){$\bullet$}
\drawvertex(C){$\bullet$}\drawvertex(D){$\bullet$}
\drawvertex(E){$\bullet$}\drawvertex(F){$\bullet$}
\drawvertex(G){$\bullet$}\drawvertex(H){$\bullet$}
\drawvertex(I){$\bullet$}

\drawundirectededge(E,C){}\drawundirectededge(B,I){}\drawundirectededge(H,F){}

\drawundirectededge(A,B){}\drawundirectededge(B,C){}\drawundirectededge(C,D){}
\drawundirectededge(D,E){}\drawundirectededge(E,F){}
\drawundirectededge(F,G){}
\drawundirectededge(H,I){}\drawundirectededge(I,A){}\drawundirectededge(G,H){}


\letvertex a=(240,10)\letvertex b=(260,44)
\letvertex c=(280,78)\letvertex d=(300,112)
\letvertex e=(320,78)\letvertex f=(340,44)
\letvertex g=(360,10)\letvertex h=(320,10)\letvertex i=(280,10)

\letvertex zz=(280,95)\letvertex yy=(320,95)\drawundirectededge(zz,yy){}

\letvertex ZZ=(240,27)\letvertex YY=(280,27) \drawundirectededge(YY,ZZ){}

\letvertex uu=(320,35)\letvertex vv=(340,0) \drawundirectededge(uu,vv){}

\drawvertex(a){$\bullet$}\drawvertex(b){$\bullet$}
\drawvertex(c){$\bullet$}\drawvertex(d){$\bullet$}
\drawvertex(e){$\bullet$}\drawvertex(f){$\bullet$}
\drawvertex(g){$\bullet$}\drawvertex(h){$\bullet$}
\drawvertex(i){$\bullet$}

\drawundirectededge(e,c){}\drawundirectededge(b,i){}\drawundirectededge(h,f){}

\drawundirectededge(a,b){}\drawundirectededge(b,c){}\drawundirectededge(c,d){}
\drawundirectededge(d,e){}\drawundirectededge(e,f){}
\drawundirectededge(f,g){}
\drawundirectededge(h,i){}\drawundirectededge(i,a){}\drawundirectededge(g,h){}
\end{picture}
\end{center}

By symmetry, each one of the pictures above gives to $Q_{n+1}$ a
contribution of
$$
abc\left(U^2_{n}(R_{n}+L_{n})+R_{n}^2(U_{n}+L_{n})+L_{n}^2(U_{n}+R_{n})\right).
$$
Consider now the following pictures.
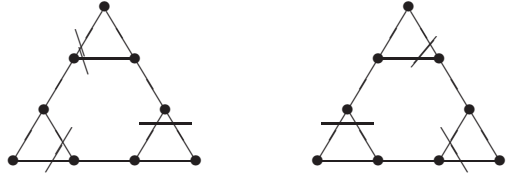
\begin{center}
\begin{picture}(400,110)
\letvertex A=(40,10)\letvertex B=(60,44)
\letvertex C=(80,78)\letvertex D=(100,112)
\letvertex E=(120,78)\letvertex F=(140,44)
\letvertex G=(160,10)\letvertex H=(120,10)\letvertex I=(80,10)

\letvertex z=(80,100)\letvertex y=(90,65)\letvertex Z=(80,35)
\letvertex Y=(60,0)\letvertex K=(120,35)\letvertex k=(160,35)

\drawundirectededge(z,y){} \drawundirectededge(Y,Z){}
\drawundirectededge(k,K){}

\drawvertex(A){$\bullet$}\drawvertex(B){$\bullet$}
\drawvertex(C){$\bullet$}\drawvertex(D){$\bullet$}
\drawvertex(E){$\bullet$}\drawvertex(F){$\bullet$}
\drawvertex(G){$\bullet$}\drawvertex(H){$\bullet$}
\drawvertex(I){$\bullet$}

\drawundirectededge(E,C){}\drawundirectededge(B,I){}\drawundirectededge(H,F){}

\drawundirectededge(A,B){}\drawundirectededge(B,C){}\drawundirectededge(C,D){}
\drawundirectededge(D,E){}\drawundirectededge(E,F){}
\drawundirectededge(F,G){}
\drawundirectededge(H,I){}\drawundirectededge(I,A){}\drawundirectededge(G,H){}

\letvertex a=(240,10)\letvertex b=(260,44)
\letvertex c=(280,78)\letvertex d=(300,112)
\letvertex e=(320,78)\letvertex f=(340,44)
\letvertex g=(360,10)\letvertex h=(320,10)\letvertex i=(280,10)

\letvertex u=(320,95)\letvertex U=(300,70)\letvertex v=(240,35)
\letvertex V=(280,35)\letvertex p=(320,35)\letvertex P=(340,0)

\drawundirectededge(u,U){}
\drawundirectededge(v,V){}\drawundirectededge(p,P){}

\drawvertex(a){$\bullet$}\drawvertex(b){$\bullet$}
\drawvertex(c){$\bullet$}\drawvertex(d){$\bullet$}
\drawvertex(e){$\bullet$}\drawvertex(f){$\bullet$}
\drawvertex(g){$\bullet$}\drawvertex(h){$\bullet$}
\drawvertex(i){$\bullet$}

\drawundirectededge(b,c){}\drawundirectededge(e,f){}
\drawundirectededge(h,i){}

 \drawundirectededge(a,b){}\drawundirectededge(c,d){}
\drawundirectededge(d,e){}\drawundirectededge(e,c){}
\drawundirectededge(f,g){}\drawundirectededge(b,i){}\drawundirectededge(h,f){}
\drawundirectededge(i,a){}\drawundirectededge(g,h){}
\end{picture}
\end{center}

They give a contribution equal to $2abcU_{n}R_{n}L_{n}$ to
$Q_{n+1}$. Now look at the following configurations.

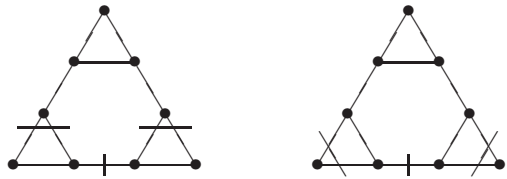
\begin{center}
\begin{picture}(400,110)
\letvertex A=(40,10)\letvertex B=(60,44)
\letvertex C=(80,78)\letvertex D=(100,112)
\letvertex E=(120,78)\letvertex F=(140,44)
\letvertex G=(160,10)\letvertex H=(120,10)\letvertex I=(80,10)

\letvertex z=(40,35)\letvertex y=(80,35)\letvertex Z=(120,35)
\letvertex Y=(160,35)\letvertex K=(100,20)\letvertex k=(100,0)

\drawundirectededge(z,y){} \drawundirectededge(Y,Z){}
\drawundirectededge(k,K){}

\drawvertex(A){$\bullet$}\drawvertex(B){$\bullet$}
\drawvertex(C){$\bullet$}\drawvertex(D){$\bullet$}
\drawvertex(E){$\bullet$}\drawvertex(F){$\bullet$}
\drawvertex(G){$\bullet$}\drawvertex(H){$\bullet$}
\drawvertex(I){$\bullet$}

\drawundirectededge(E,C){}\drawundirectededge(B,I){}\drawundirectededge(H,F){}

\drawundirectededge(A,B){}\drawundirectededge(B,C){}\drawundirectededge(C,D){}
\drawundirectededge(D,E){}\drawundirectededge(E,F){}
\drawundirectededge(F,G){}
\drawundirectededge(H,I){}\drawundirectededge(I,A){}\drawundirectededge(G,H){}

\letvertex a=(240,10)\letvertex b=(260,44)
\letvertex c=(280,78)\letvertex d=(300,112)
\letvertex e=(320,78)\letvertex f=(340,44)
\letvertex g=(360,10)\letvertex h=(320,10)\letvertex i=(280,10)

\letvertex u=(240,35)\letvertex U=(260,0)\letvertex v=(340,0)
\letvertex V=(360,35)\letvertex p=(300,20)\letvertex P=(300,0)

\drawundirectededge(u,U){}
\drawundirectededge(v,V){}\drawundirectededge(p,P){}

\drawvertex(a){$\bullet$}\drawvertex(b){$\bullet$}
\drawvertex(c){$\bullet$}\drawvertex(d){$\bullet$}
\drawvertex(e){$\bullet$}\drawvertex(f){$\bullet$}
\drawvertex(g){$\bullet$}\drawvertex(h){$\bullet$}
\drawvertex(i){$\bullet$}

\drawundirectededge(b,c){}\drawundirectededge(e,f){}
\drawundirectededge(h,i){}

 \drawundirectededge(a,b){}\drawundirectededge(c,d){}
\drawundirectededge(d,e){}\drawundirectededge(e,c){}
\drawundirectededge(f,g){}\drawundirectededge(b,i){}\drawundirectededge(h,f){}
\drawundirectededge(i,a){}\drawundirectededge(g,h){}
\end{picture}
\end{center}

By symmetry, they correspond to the contribution
$2T_{n}(abU_{n}R_{n}+ bcU_{n}L_{n}+acR_{n}L_{n})$. The following
picture gives the contribution
$T_{n}\left(bU_{n}^2(a+c)+aR_{n}^2(b+c)+cL_{n}^2(a+b)\right)$.

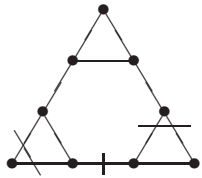
\begin{center}
\begin{picture}(400,110)
\letvertex A=(140,10)\letvertex B=(160,44)
\letvertex C=(180,78)\letvertex D=(200,112)
\letvertex E=(220,78)\letvertex F=(240,44)
\letvertex G=(260,10)\letvertex H=(220,10)\letvertex I=(180,10)

\letvertex z=(140,35)\letvertex y=(160,0)
\drawundirectededge(z,y){}
\letvertex Z=(220,35)\letvertex
Y=(260,35) \drawundirectededge(Y,Z){}
\letvertex u=(200,20)\letvertex
v=(200,0) \drawundirectededge(u,v){}

\drawvertex(A){$\bullet$}\drawvertex(B){$\bullet$}
\drawvertex(C){$\bullet$}\drawvertex(D){$\bullet$}
\drawvertex(E){$\bullet$}\drawvertex(F){$\bullet$}
\drawvertex(G){$\bullet$}\drawvertex(H){$\bullet$}
\drawvertex(I){$\bullet$}

\drawundirectededge(E,C){}\drawundirectededge(B,I){}\drawundirectededge(H,F){}
\drawundirectededge(A,B){}\drawundirectededge(B,C){}\drawundirectededge(C,D){}
\drawundirectededge(D,E){}\drawundirectededge(E,F){}
\drawundirectededge(F,G){}
\drawundirectededge(H,I){}\drawundirectededge(I,A){}\drawundirectededge(G,H){}
\end{picture}
\end{center}
Finally, we have to consider the following four situations.
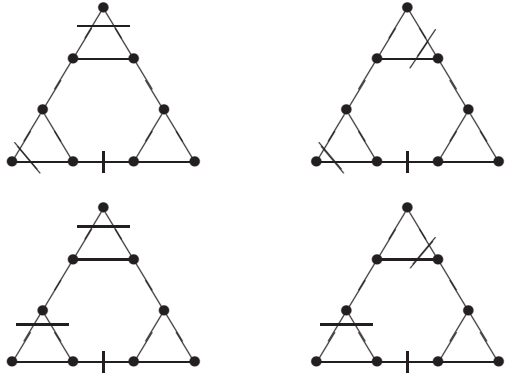
\begin{center}
\begin{picture}(400,110)
\letvertex A=(40,10)\letvertex B=(60,44)
\letvertex C=(80,78)\letvertex D=(100,112)
\letvertex E=(120,78)\letvertex F=(140,44)
\letvertex G=(160,10)\letvertex H=(120,10)\letvertex I=(80,10)

\letvertex z=(80,100)\letvertex y=(120,100)\letvertex Z=(40,25)
\letvertex Y=(60,0)\letvertex K=(100,20)\letvertex k=(100,0)

\drawundirectededge(z,y){} \drawundirectededge(Y,Z){}
\drawundirectededge(k,K){}

\drawvertex(A){$\bullet$}\drawvertex(B){$\bullet$}
\drawvertex(C){$\bullet$}\drawvertex(D){$\bullet$}
\drawvertex(E){$\bullet$}\drawvertex(F){$\bullet$}
\drawvertex(G){$\bullet$}\drawvertex(H){$\bullet$}
\drawvertex(I){$\bullet$}

\drawundirectededge(E,C){}\drawundirectededge(B,I){}\drawundirectededge(H,F){}
\drawundirectededge(A,B){}\drawundirectededge(B,C){}\drawundirectededge(C,D){}
\drawundirectededge(D,E){}\drawundirectededge(E,F){}
\drawundirectededge(F,G){}
\drawundirectededge(H,I){}\drawundirectededge(I,A){}\drawundirectededge(G,H){}

\letvertex a=(240,10)\letvertex b=(260,44)
\letvertex c=(280,78)\letvertex d=(300,112)
\letvertex e=(320,78)\letvertex f=(340,44)
\letvertex g=(360,10)\letvertex h=(320,10)\letvertex i=(280,10)

\letvertex u=(320,100)\letvertex U=(300,70)\letvertex v=(240,25)
\letvertex V=(260,0)\letvertex p=(300,20)\letvertex P=(300,0)

\drawundirectededge(u,U){}
\drawundirectededge(v,V){}\drawundirectededge(p,P){}

\drawvertex(a){$\bullet$}\drawvertex(b){$\bullet$}
\drawvertex(c){$\bullet$}\drawvertex(d){$\bullet$}
\drawvertex(e){$\bullet$}\drawvertex(f){$\bullet$}
\drawvertex(g){$\bullet$}\drawvertex(h){$\bullet$}
\drawvertex(i){$\bullet$}

\drawundirectededge(b,c){}\drawundirectededge(e,f){}
\drawundirectededge(h,i){}

\drawundirectededge(a,b){}\drawundirectededge(c,d){}
\drawundirectededge(d,e){}\drawundirectededge(e,c){}
\drawundirectededge(f,g){}\drawundirectededge(b,i){}\drawundirectededge(h,f){}
\drawundirectededge(i,a){}\drawundirectededge(g,h){}
\end{picture}
\end{center}

\begin{center}
\begin{picture}(400,110)
\letvertex A=(40,10)\letvertex B=(60,44)
\letvertex C=(80,78)\letvertex D=(100,112)
\letvertex E=(120,78)\letvertex F=(140,44)
\letvertex G=(160,10)\letvertex H=(120,10)\letvertex I=(80,10)

\letvertex z=(80,100)\letvertex y=(120,100)\letvertex Z=(40,35)
\letvertex Y=(80,35)\letvertex K=(100,20)\letvertex k=(100,0)

\drawundirectededge(z,y){} \drawundirectededge(Y,Z){}
\drawundirectededge(k,K){}

\drawvertex(A){$\bullet$}\drawvertex(B){$\bullet$}
\drawvertex(C){$\bullet$}\drawvertex(D){$\bullet$}
\drawvertex(E){$\bullet$}\drawvertex(F){$\bullet$}
\drawvertex(G){$\bullet$}\drawvertex(H){$\bullet$}
\drawvertex(I){$\bullet$}

\drawundirectededge(E,C){}\drawundirectededge(B,I){}\drawundirectededge(H,F){}
\drawundirectededge(A,B){}\drawundirectededge(B,C){}\drawundirectededge(C,D){}
\drawundirectededge(D,E){}\drawundirectededge(E,F){}
\drawundirectededge(F,G){}
\drawundirectededge(H,I){}\drawundirectededge(I,A){}\drawundirectededge(G,H){}

\letvertex a=(240,10)\letvertex b=(260,44)
\letvertex c=(280,78)\letvertex d=(300,112)
\letvertex e=(320,78)\letvertex f=(340,44)
\letvertex g=(360,10)\letvertex h=(320,10)\letvertex i=(280,10)

\letvertex u=(320,95)\letvertex U=(300,70)\letvertex v=(240,35)
\letvertex V=(280,35)\letvertex p=(300,20)\letvertex P=(300,0)

\drawundirectededge(u,U){}
\drawundirectededge(v,V){}\drawundirectededge(p,P){}

\drawvertex(a){$\bullet$}\drawvertex(b){$\bullet$}
\drawvertex(c){$\bullet$}\drawvertex(d){$\bullet$}
\drawvertex(e){$\bullet$}\drawvertex(f){$\bullet$}
\drawvertex(g){$\bullet$}\drawvertex(h){$\bullet$}
\drawvertex(i){$\bullet$}

\drawundirectededge(b,c){}\drawundirectededge(e,f){}
\drawundirectededge(h,i){}\drawundirectededge(a,b){}\drawundirectededge(c,d){}
\drawundirectededge(d,e){}\drawundirectededge(e,c){}
\drawundirectededge(f,g){}\drawundirectededge(b,i){}\drawundirectededge(h,f){}
\drawundirectededge(i,a){}\drawundirectededge(g,h){}
\end{picture}
\end{center}
The first and the third pictures provide each a contribution equal
to $T_{n}(U_{n}R_{n}(a+b)c + U_{n}L_{n}(b+c)a+R_{n}L_{n}(a+c)b)$.
The second one gives
$T_{n}\left(bU_{n}^2(a+c)+aR_{n}^2(b+c)+cL_{n}^2(a+b)\right) $.
The fourth one gives $2T_{n}(abU_{n}R_{n} +acR_{n}L_{n}+bc
U_{n}L_{n})$. This completes the proof.
\end{proof}

Equations \eqref{Tbrutto}, \eqref{Ubrutto}, \eqref{Rbrutto},
\eqref{Lbrutto} and \eqref{Qbrutto}, with the initial conditions
given in Theorem \ref{Teoremabrutti}, seem to be very hard to be
explicitly solved. For this reason, in the next section we use a
different strategy to find the weighted generating function. On
the other hand, if we are interested in the complexity of these
graphs, we can evaluate each generating function in $a=b=c=1$ and
we get simpler equations. Let us define:
\begin{itemize}
\item $\tau_n:=\tau(\Sigma_n)=T_n(1,1,1)=$ complexity of
$\Sigma_n$;
\item $s_n:=U_n(1,1,1) = R_n(1,1,1)=L_n(1,1,1)=$ number of spanning
$2$-forests, where two fixed outmost vertices are in the same
connected component and the third one lies in a different
component;
\item $q_n:=Q_n(1,1,1) =$ number of spanning $3$-forests, where each component contains exactly one outmost
vertex.
\end{itemize}

\begin{cor}\label{relations}
For each $n\geq 1$, the values $\tau_n, s_n$ and $q_n$ satisfy the
following relations:
\begin{eqnarray*}
\tau_{n+1} = 3\tau_{n}^3+6\tau_{n}^2s_{n}
\end{eqnarray*}
\begin{eqnarray*}
s_{n+1} = \tau_{n}^3 + 7\tau_{n}^2s_{n}
+7\tau_{n}s_{n}^2+\tau_{n}^2q_{n}
\end{eqnarray*}
\begin{eqnarray*}
q_{n+1} &=& 3\tau_{n}^2q_{n} + 12\tau_{n}s_{n}q_{n} +14s_{n}^3 \\
&+& 12\tau_{n}^2s_{n}+ \tau_{n}^3+36\tau_{n}s_{n}^2\nonumber,
\end{eqnarray*}
with initial conditions
$$
\tau_1=3 \qquad s_1=q_1=1.
$$
\end{cor}
\begin{proof}
The proof follows from Theorem \ref{Teoremabrutti}, by evaluating
each function for $a=b=c=1$ and recalling that
$U_n(1,1,1)=R_n(1,1,1)=L_n(1,1,1)$, for each $n\geq 1$.
\end{proof}

\begin{prop}
For every $n\geq 1$, the values $\tau_n, s_n$ and $q_n$ satisfying
the relations given in Corollary \ref{relations} are:
\begin{enumerate}
\item $\tau_n = 3^{\frac{3^n+2n-1}{4}}\cdot
5^{\frac{3^n-2n-1}{4}}$;
\item $s_n = 3^{\frac{3^n-2n-1}{4}}\cdot 5^{\frac{3^n-2n-1}{4}}\cdot
\frac{5^n-3^n}{2}$;
\item $q_n = 3^{\frac{3^n-6n+3}{4}}\cdot 5^{\frac{3^n-2n-1}{4}}\cdot \left(\frac{5^n-3^n}{2}\right)^2$
\end{enumerate}
In particular, the asymptotic growth constant of the spanning
trees of $\Sigma_n$ is $\frac{1}{4}\left(\log 3 + \log 5\right)$.
\end{prop}

\begin{proof}
The proof can be given by induction on $n$. Then, the asymptotic
growth constant is obtained as the limit
$$
\lim_{n\to \infty}\frac{\log(\tau_n)}{|V(\Sigma_n)|},
$$
where $|V(\Sigma_n)| = 3^n$ is the number of vertices of
$\Sigma_n$, for each $n\geq 1$.
\end{proof}

\subsection{Computation of the weighted generating function}

In this section, we compute the weighted generating function of
the spanning trees on the graph $\Sigma_n$ by using the weighted
version of the Kirchhoff's Theorem. The idea is to use the
self-similar presentation of the generators of the group in order
to describe recursively the adjacency matrix of the graph. The
Laplace matrix of $\Sigma_n$ is then obtained as the difference
between the degree matrix of $\Sigma_n$ and its adjacency matrix.
By using the Schur complement Formula, one can compute a cofactor
of the Laplace matrix passing from a square matrix of size $3^n$
to a square matrix of size $3^{n-1}$, which turns out to have the
same structure of the original matrix, where the entries have been
transformed via a rational function. This gives rise to a
recursion process allowing to compute the cofactor. The same
strategy is used in \cite{noidimeri} to compute the partition
function of the dimer model. See also \cite{hanoi}, where the
authors use the same idea to study the spectrum of the group
$H^{(3)}$.

Let $\Delta_n$ be the adjacency matrix of $\Sigma_n$, whose rows
and columns are indexed by the vertices of $\Sigma_n$, i.e., words
of length $n$ in the alphabet $\{0,1,2\}$, which are ordered
lexicografically. Hence, $\Delta_n=a_n+b_n+c_n$, where the
matrices $a_n,b_n$ and $c_n$ describe the action of the generators
$a,b$ and $c$ of $H^{(3)}$, respectively, on the $n$-th level of
the rooted ternary tree. In other words, we have:
$$
(a_n)_{ij}=
\begin{cases}
a & \text{if} \ a(v_i)=v_j\\
 0  & \text{otherwise},
\end{cases}
$$
where $v_i$, for $i=1,\ldots, 3^n$, denotes the $i$-th vertex of
$\Sigma_n$ with respect to the lexicografic order (similarly for
$b_n$ and $c_n$). Notice that, in this way, we are writing the
adjacency matrix of the graph with loops. However, in the Laplace
matrix, the entries corresponding to loops will be deleted by
using the degree matrix $(a+b+c)I_n$, according with the fact that
a spanning tree cannot contain any loop.

The matrices $a_n$, $b_n$ and $c_n$ can be recursively
represented, by using the self-similar description of the
generators $a,b$ and $c$ of the group. For $n=1$, we set
$$
a_1 = \begin{pmatrix}
  0 & a & 0 \\
  a & 0 & 0 \\
  0 & 0 & a
\end{pmatrix} \qquad b_1 =\begin{pmatrix}
  0 & 0 & b \\
  0 & b & 0 \\
  b & 0 & 0
\end{pmatrix} \qquad  c_1 =  \begin{pmatrix}
  c & 0 & 0 \\
  0 & 0 & c \\
  0 & c & 0
\end{pmatrix}
$$
and, for every $n>1$, we put
$$
a_n = \begin{pmatrix}
  0 & aI_{n-1} & 0 \\
  aI_{n-1} & 0 & 0 \\
  0 & 0 & a_{n-1}
\end{pmatrix} \qquad b_n =\begin{pmatrix}
  0 & 0 & bI_{n-1} \\
  0 & b_{n-1} & 0 \\
  bI_{n-1} & 0 & 0
\end{pmatrix} \qquad c_n =  \begin{pmatrix}
  c_{n-1} & 0 & 0 \\
  0 & 0 & cI_{n-1} \\
  0 & cI_{n-1} & 0
\end{pmatrix}.
$$
Therefore, Kirchhoff's Theorem states that the weighted generating
function of the spanning trees on $\Sigma_n$ can be obtained by
computing any cofactor of the Laplace matrix
$$
(a+b+c)I_{n}-\Delta_n = \begin{pmatrix}
  (a+b+c)I_{n-1} -c_{n-1} & -aI_{n-1} & -bI_{n-1} \\
  -aI_{n-1} &(a+b+c)I_{n-1}- b_{n-1} & -cI_{n-1} \\
  -bI_{n-1} & -cI_{n-1} &(a+b+c)I_{n-1}- a_{n-1}
\end{pmatrix}.
$$
We can choose, for instance, to compute the cofactor associated
with the first row and the first column of the Laplace matrix. In
order to compute it we put, for every $n>1$,
$$
\overline{\Delta}_n= \begin{pmatrix}
  c_{n-1} & aI_{n-1}^0 & bI_{n-1}^0 \\
  aI_{n-1}^0 & b_{n-1} & cI_{n-1} \\
  bI_{n-1}^0 & cI_{n-1} & a_{n-1}
\end{pmatrix},
$$
with
$$
I_n^0=I_n-\begin{pmatrix}
  1 & 0 & \cdots & 0 \\
  0 & 0 & \cdots & 0 \\
  \vdots & 0 & \ddots & 0 \\
  0 & 0 & 0 & 0
\end{pmatrix}.
$$
Moreover, we fix the notation $\gamma \widetilde{I}_{k}:= \gamma
I_k^0+(a+b+c)(I_k-I_k^0)$, for each $k\geq 1$.

Define
$$
\Lambda_n:=\!(a+b+c)I_n-\overline{\Delta}_n \!=\!\begin{pmatrix}
  (a+b+c)I_{n-1}-c_{n-1} & -aI_{n-1}^0 & -bI_{n-1}^0 \\
  -aI_{n-1}^0 & (a+b+c)I_{n-1}-b_{n-1} & -cI_{n-1} \\
  -bI_{n-1}^0 & -cI_{n-1} & (a+b+c)I_{n-1}-a_{n-1}
\end{pmatrix}\!.
$$
The introduction of the matrices $I_n^0$ guarantees that
$\frac{\det(\Lambda_n)}{a+b}$ is the generating function of the
spanning trees of $\Sigma_n$, because we have performed all the
necessary cancellations in $\overline{\Delta}_n$. More precisely,
since $(\Lambda_n)_{11}=a+b$ and this is the only non-zero entry
of the first row and column of $\Lambda_n$, it turns out that
$\frac{\det(\Lambda_n)}{a+b}$ is equal to the cofactor of
$(a+b+c)I_n-\Delta_n$ associated with the first row and the first column.\\
\indent Moreover, we define the rational function
$P:\mathbb{R}^9\longrightarrow \mathbb{R}^9$ as
$P(\underline{x})=(P_1(\underline{x}), \ldots,
P_9(\underline{x}))$, where
$$
P_1(\underline{x})=x_1, \qquad \ P_2(\underline{x})=x_2, \qquad
P_3(\underline{x})=x_3,
$$
\noindent \begin{eqnarray*}
P_4(\underline{x})&=& 1/D \\
&\cdot &\left(
x_2x_3x_4^2x_5x_6x_9+x_1x_4x_5^2x_7x_8^2-x_1x_4^3x_5^2x_8-x_1x_2^2x_4x_5^2x_7+x_2x_3x_4^3x_9^2+x_1x_4^2x_5x_6x_7x_8\right.\\&-&x_1x_4^4x_5x_6
-x_1^2x_2x_3x_4^3+x_1x_5x_6x_7^2x_8^2-x_1x_3^2x_5x_6x_8^2-x_1x_4^2x_5x_6x_7x_8-x_1x_2^2x_5x_6x_7^2\\&+&\left.x_1x_2^2x_3^2x_5x_6+x_2x_3x_4x_5^2x_6^2
-x_1x_3^2x_4x_6^2x_8+x_1x_4x_6^2x_7^2x_8-x_1x_4^3x_6^2x_7+x_2x_3x_4^2x_5x_6x_9\right)
\end{eqnarray*}
\begin{eqnarray*}
P_5(\underline{x})&=& 1/D \\
&\cdot &\left(
x_1x_3x_5^3x_8^2-x_1x_2^2x_3x_5^3-x_2x_4x_5^4x_6+x_2x_4x_5^2x_6x_7x_9+x_2x_4^2x_5x_7x_9^2-x_2x_4^2x_5^3x_9\right.\\&+&x_1x_3x_4x_5^2x_6x_8-
x_1^2x_2x_4^2x_5x_7+x_2x_5x_6^2x_7^2x_9-x_2x_3^2x_5x_6^2x_9+x_1x_3x_4x_5^2x_6x_8-x_2x_5^3x_6^2x_7\\&+&\left.x_2x_4x_6x_7^2x_9^2-
x_2x_3^2x_4x_6x_9^2-x_2x_4x_5^2x_6x_7x_9-x_1^2x_2x_4x_6x_7^2+x_1x_3x_4^2x_5x_6^2+x_1^2x_2x_3^2x_4x_6\right)
\end{eqnarray*}
\begin{eqnarray*}
P_6(\underline{x})&=& 1/D \\
&\cdot &\left(
 x_3x_4x_5x_8^2x_9^2-x_2^2x_3x_4x_5x_9^2-x_1^2x_3x_4x_5x_8^2+x_1^2x_2^2x_3x_4x_5+x_1x_2x_4^2x_5^2x_6+2x_1x_2x_4x_5x_6^2x_7\right.\\
 &+&x_3x_5^2x_6x_8^2x_9-x_3x_5^2x_6^3x_8-x_2^2x_3x_5^2x_6x_9-x_1^2x_3x_4^2x_6x_8+x_3x_4^2x_6x_8x_9^2-x_3x_4^2x_6^3x_9\\
 &-&\left.x_3x_4x_5x_6^4+x_1x_2x_6^3x_7^2-x_1x_2x_3^2x_6^3
\right)
\end{eqnarray*}
\begin{eqnarray*}
P_7(\underline{x})&=& x_7+ 1/D\\
&\cdot &\left(-x_5x_7^2x_8^2x_9+x_3^2x_5x_7^2x_9+x_2^2x_5x_7^2x_9-x_2^2x_3^2x_5x_9+x_5^3x_7x_8^2-x_4^2x_5^3x_8-x_2^2x_5^3x_7-x_4x_7^2x_8x_9^2\right.\\
&+&x_4^2x_5x_7x_8x_9\!+\!x_3^2x_4x_8x_9^2\!+\!x_4^3x_7x_9^2\!-\!x_4^3x_5^2x_9\!-\!x_1^2x_3^2x_4x_8\!+\!x_1^2x_4x_7^2x_8\!-\!x_1^2x_4^3x_7
\!+\!2x_3^2x_4x_5x_6x_8x_9\\
&-&\left.2x_4x_5x_6x_7^2x_8x_9+2x_4^3x_5x_6x_7x_9+2x_4x_5^3x_6x_7x_8-2x_4^3x_5^3x_6
-2x_1x_2x_3x_4^2x_5^2+x_4x_5^2x_7x_8x_9\right)
\end{eqnarray*}
\begin{eqnarray*}
P_8(\underline{x})&=& x_8+ 1/D\\
&\cdot &\left(
-x_4^2x_6^4x_7-x_3^2x_6^4x_8-2x_4^3x_5x_6^3-x_4^4x_6^2x_9-x_1^2x_4^4x_8+x_4^4x_8x_9^2+x_2^2x_4^2x_7x_9^2+x_3^2x_6^2x_8^2x_9\right.\\
&-&x_2^2x_3^2x_6^2x_9+2x_4^2x_6^2x_7x_8x_9\!-\!x_1^2x_2^2x_4^2x_7\!+\!x_1^2x_4^2x_7x_8^2\!-\!x_4^2x_7x_8^2x_9^2\!-\!2x_1x_2x_3x_4^2x_6^2
\!+\!2x_4^3x_5x_6x_8x_9\\
&-&\left.2x_4x_5x_6x_7x_8^2x_9+2x_4x_5x_6^3x_7x_8+2x_2^2x_4x_5x_6x_7x_9-x_6^2x_7^2x_8^2x_9+x_2^2x_6^2x_7^2x_9+x_6^4x_7^2x_8\right)
\end{eqnarray*}
\begin{eqnarray*}
P_9(\underline{x})&=&x_9+ 1/D\\
&\cdot &\left(
x_5^4x_8^2x_9-x_2^2x_5^4x_9-x_5^4x_6^2x_8+x_1^2x_5^2x_7x_8^2-x_1^2x_2^2x_5^2x_7-x_5^2x_7x_8^2x_9^2+x_2^2x_5^2x_7x_9^2+2x_5^2x_6^2x_7x_8x_9\right.\\
&-&2x_1x_2x_3x_5^2x_6^2 + 2x_4x_5^3x_6x_8x_9 - 2x_4x_5^3x_6^3 +
2x_1^2x_4x_5x_6x_7x_8-2x_4x_5x_6x_7x_8x_9^2 + 2x_4x_5x_6^3x_7x_9\\
&-&\left. x_3^2x_6^4x_9- x_5^2x_6^4x_7 +x_1^2x_6^2x_7^2x_8-
x_1^2x_3^2x_6^2x_8 -x_6^2x_7^2x_8x_9^2+ x_3^2x_6^2x_8x_9^2+
x_6^4x_7^2x_9\right),
\end{eqnarray*}
with
\begin{eqnarray*}
D&=& D(x_1,\ldots, x_9)\\
&=&x_7^2x_8^2x_9^2-x_3^2x_8^2x_9^2-x_4^2x_7x_8x_9^2-x_2^2x_7^2x_9^2
+x_2^2x_3^2x_9^2-x_5^2x_7x_8^2x_9+x_3^2x_6^2x_8x_9\\&-&x_6^2x_7^2x_8x_9+
x_4^2x_5^2x_8x_9+x_4^2x_6^2x_7x_9+x_2^2x_5^2x_7x_9-x_4^2x_7^2x_8^2
+x_3^2x_4^2x_8^2\\&+&x_5^2x_6^2x_7x_8+x_4^4x_7x_8-2x_1x_2x_3x_4x_5x_6-x_1^2x_2^2x_3^2+x_1^2x_2^2x_7^2-x_4^2x_5^2x_6^2.
\end{eqnarray*}
Denote $P^{(k)}(\underline{x}):=P^{(k)}(x_1,\ldots, x_9)$ the
$k$-th iteration of the function $P$ and use the notation
$$
P^{(k)}(a,b,c,a,b,c,a+b+c,a+b+c,a+b+c):=\left(a,b,c,a^{(k)},
b^{(k)}, c^{(k)},e^{(k)},f^{(k)},g^{(k)}\right),
$$
and
$$
D^{(k)}:=D^{(k)}(a,b,c,a,b,c,a+b+c,a+b+c,a+b+c)= D
\left(a,b,c,a^{(k)},b^{(k)},c^{(k)},e^{(k)},f^{(k)},g^{(k)}\right),
$$
where $a^{(0)}=a,b^{(0)}=b,c^{(0)}=c$ and
$e^{(0)}=f^{(0)}=g^{(0)}=a+b+c$.

Moreover we define, for each $\underline{x}\in \mathbb{R}^9$,
$$
\Lambda_{k}(\underline{x})= \begin{pmatrix}
 P_7(\underline{x})\widetilde{I}_{k-1}- c_{k-1} & -P_4(\underline{x})I_{k-1}^0 & -P_5(\underline{x})I_{k-1}^0 \\
 -P_4(\underline{x})I_{k-1}^0 & P_8(\underline{x})I_{k-1}-b_{k-1} & -P_6(\underline{x})I_{k-1} \\
  -P_5(\underline{x})I_{k-1}^0 & -P_6(\underline{x})I_{k-1} & P_9(\underline{x})I_{k-1}-a_{k-1}
\end{pmatrix}.
$$
\begin{teo}\label{PROPOSITIONPARTITION}
For each $n\geq 3$, the weighted generating function $T_n(a,b,c)$
of the spanning trees on the Schreier graph $\Sigma_n$ of the
Hanoi Towers group $H^{(3)}$ is
\begin{eqnarray*}
T_n(a,b,c)=\frac{1}{a+b} \prod_{k=0}^{n-3}
(D^{(k)})^{3^{n-k-2}}\cdot
\det\left(\Lambda_2\left(P^{(n-3)}(a,b,c,a,b,c,a+b+c,a+b+c,a+b+c)\right)\right)
\end{eqnarray*}
\end{teo}

\begin{proof}
It is clear, from what said above, that
$T_n(a,b,c)=\frac{\det(\Lambda_n)}{a+b}$. More precisely, the
factor $a+b$ corresponds to the entry $(1,1)$ of $\Lambda_n$. This
entry must be simplified, since we want to compute the associated
cofactor. If we expand twice the matrix $\Lambda_n$, using
recursion, and we perform the transpositions $(17)$ and $(58)$ for
both rows and columns, we get the matrix
$$
\begin{pmatrix}
  M_{11} & M_{12} \\
  M_{21} & M_{22}
\end{pmatrix}=
$$\tiny
$$
 \left(\begin{array}{cccccc|ccc}
  a+b+c & 0 & 0 & -cI_{n-2} & -aI_{n-2} & 0 & -bI_{n-2}^0 & 0 & 0 \\
  0 & a+b+c & -cI_{n-2} & 0 & -bI_{n-2} & 0 & 0 & -aI_{n-2} & 0 \\
  0 & -cI_{n-2} & a+b+c & 0 & 0 & -aI_{n-2} & 0 & 0 & -bI_{n-2} \\
  -cI_{n-2} & 0 & 0 & a+b+c & 0 & -bI_{n-2} & -aI_{n-2}^0 & 0 & 0 \\
  -aI_{n-2} & -bI_{n-2} & 0 & 0 & a+b+c & 0 & 0 & -cI_{n-2} & 0 \\
  0 & 0 & -aI_{n-2} & -bI_{n-2} & 0 & a+b+c & 0 & 0 & -cI_{n-2}\\
  \hline
  -bI_{n-2}^0 & 0 & 0 & -aI_{n-2}^0 & 0 & 0 & a+b+c-c_{n-2} & 0 & 0 \\
  0 & -aI_{n-2} & 0 & 0 & -cI_{n-2} & 0 & 0 & a+b+c-b_{n-2} & 0 \\
  0 & 0 & -bI_{n-2} & 0 & 0 & -cI_{n-2} & 0 & 0 & a+b+c-a_{n-2}
\end{array}\right).
$$\normalsize
Note that each entry is a square matrix of size $3^{n-2}$ and
$a+b+c$ is multiplied by $I_{n-2}$. Hence, the Schur complement
Formula gives
\begin{eqnarray}\label{matrice per hanoi}
\det(\Lambda_n) &=& \det(M_{11})\cdot
\det(M_{22}-M_{21}M_{11}^{-1}M_{12})\\
&=& (D^{(0)})^{3^{n-2}}\cdot\det\begin{pmatrix}
  e^{(1)}\widetilde{I}_{n-2}-c_{n-2} & -a^{(1)}I_{n-2}^0 & -b^{(1)}I_{n-2}^0 \\
 -a^{(1)}I_{n-2}^0 & f^{(1)}I_{n-2}-b_{n-2} & -c^{(1)}I_{n-2} \\
  -b^{(1)}I_{n-2}^0 &  -c^{(1)}I_{n-2}  & g^{(1)}I_{n-2}-a_{n-2}
\end{pmatrix}.\nonumber
\end{eqnarray}
The fundamental remark is that the matrix in Equation
\eqref{matrice per hanoi} has the same shape as $\Lambda_n$, since
$e^{(0)}\widetilde{I}_{n-1}=(a+b+c)I_{n-1}$. Therefore, we can
apply a recursive argument and use the same strategy $n-3$ times
until we get a $9\times 9$ matrix, that coincides with
$\Lambda_2\left(P^{(n-3)}(a,b,c,a,b,c,a+b+c,a+b+c,a+b+c)\right)$.

Observe that the entry $(1,1)$ of
$\Lambda_2\left(P^{(n-3)}(a,b,c,a,b,c,a+b+c,a+b+c,a+b+c)\right)$
is still equal to $a+b$; moreover, all the remaining entries in
the first row and column of this matrix are zero. It follows that
in the determinant of
$\Lambda_2\left(P^{(n-3)}(a,b,c,a,b,c,a+b+c,a+b+c,a+b+c)\right)$ a
factor $(a+b)$ occurs, that we have to simplify in order to apply
Kirchhoff's Theorem. Hence, Equation (\ref{matrice per hanoi})
becomes
\begin{eqnarray*}
\det(\Lambda_n)&=&(D^{(0)})^{3^{n-2}}\cdot\det(\Lambda_{n-1}(a,b,c,a,b,c,a+b+c,a+b+c,a+b+c))\\
&=& (D^{(0)})^{3^{n-2}}\cdot (D^{(1)})^{3^{n-3}}\cdot\det\left(
\Lambda_{n-2}(P^{(1)}(a,b,c,a,b,c,a+b+c,a+b+c,a+b+c))\right)\\
&=& \prod_{k=0}^{n-3}(D^{(k)})^{3^{n-k-2}}\cdot \det\left(\Lambda_2\left(P^{(n-3)}(a,b,c,a,b,c,a+b+c,a+b+c,a+b+c)\right)\right) \\
&=& (a+b)T_n(a,b,c).
\end{eqnarray*}
\end{proof}


\section{Some statistics}\label{statistics}

In this section, we deal with a statistical analysis about the
number of edges, with a fixed label $w\in\{a,b,c\}$, occurring in
a random spanning tree of the considered graph.\\ \indent It is
clear that in the case of the Schreier graphs $\{\Sigma_n\}_{n\geq
1}$ of the Hanoi Towers group, as in the case of both directional
and Schreier labellings of the Sierpi\'{n}ski graphs
$\{\Gamma_n\}_{n\geq 1}$, all the weights play
the same role in the labelling of the edges of the graph.\\
\indent On the other hand, in the rotational-invariant model, the
weights $a$ and $b$ are symmetric, whereas the weight $c$ plays a
special role in the labelling of the graph. Hence, it is
interesting to perform such analysis on the Sierpi\'{n}ski graphs
$\{\Gamma_n\}_{n\geq 1}$, when the edges are endowed with the
rotational-invariant labelling.

Our techniques are classical: more precisely, logarithmic
derivatives of the weighted generating function $T_n(a,b,c)$ with
respect to $w$ give us the mean density of $w$-edges in a random
spanning tree. We can further find the variance and show that the
limiting distribution is normal.

Let $w_n$ be the random variable given by the number of edges
labelled $w$ in a random spanning tree of $\Gamma_n$, with
$w\in\{a,b,c\}$. Denote by $\mu_{n,w}$ and $\sigma^2_{n,w}$ the
mean and the variance of $w_n$, respectively. From the remark
above it follows that:
$$
\mu_{n,a}=\mu_{n,b} \qquad \qquad \sigma^2_{n,a}=\sigma^2_{n,b}.
$$

\begin{prop}\label{propstat}
\begin{enumerate}
\item The means and the variances of the random variables
$a_n,b_n$ and $c_n$ are:
$$
\mu_{n,a}=\mu_{n,b}=\frac{16\cdot 3^n+7}{30} \qquad \qquad
\mu_{n,c}=\frac{13\cdot 3^n+1}{30}
$$
$$
\sigma^2_{n,a}=\sigma^2_{n,b}=\frac{199\cdot 3^n+28}{900} \qquad
\qquad \sigma^2_{n,c}=\frac{34\cdot 3^n-2}{225}.
$$
\item The random variables $a_n,b_n$ and $c_n$ are asymptotically normal, as $n\rightarrow \infty$.
\end{enumerate}
\end{prop}

\begin{proof}
Let us prove the assertion for the random variable $c_n$ (similar
computations can be done for $a_n$). Take the generating function
$T_n(a,b,c)$ given in Theorem \ref{modellofacileteo} and put:
$$
T_n(c):= T_n(1,1,c) =
2^{\frac{3^n-1}{2}}3^{\frac{3^n+2n-1}{4}}5^{\frac{3^{n-1}-2n+1}{4}}(3c+2)^{\frac{3^{n-1}-1}{2}}(2c+1)^{\frac{3^n+1}{2}}.
$$
We can obtain the mean $\mu_{n,c}$ and the variance
$\sigma^2_{n,c}$ of $c_n$ by studying the derivatives of the
function $\log(T_n(c))$. We get
$$
\mu_{n,c}=\left(\log(T_n(c))\right)'\left|_{c=1} \right.=
\frac{T_n'(c)}{T_n(c)}\left|_{c=1}=\frac{13\cdot
3^n+1}{30}\right..
$$
Taking once more derivative, one gets
$$
\left(\log(T_n(c))\right)''\left|_{c=1}\right.
=\frac{T_n''(c)T_n(c)-(T_n'(c))^2}{(T_n(c))^2}\left|_{c=1}=-\frac{127\cdot
3^n+19}{450}.\right.
$$
Hence,
$$
\sigma_{n,c}^2= \left(\log(T_n(c))\right)''\left|_{c=1}\right. +
\mu_{n,c}=\frac{34\cdot 3^n-2}{225}.
$$
Next, let $C_n = \frac{c_n-\mu_{n,c}}{\sigma_{n,c}}$ be the
normalized random variable; then the moment generating function of
$C_n$ is given by
$$
\mathbb{E}(e^{tC_n}) =
e^{-\mu_{n\!,\!c}t/\sigma_{n\!,\!c}}\mathbb{E}(e^{tc_n/\sigma_{n\!,\!c}})
=
e^{-\mu_{n\!,\!c}t/\sigma_{n\!,\!c}}\frac{T_n(e^{t/\sigma_{n\!,\!c}})}{T_n(1)}.
$$
We get
$$
\mathbb{E}(e^{tC_n}) =
3^{-\frac{3^n+1}{2}}5^{-\frac{3^n-3}{6}}e^{-\frac{(13\cdot
3^n+1)t}{2(34\cdot 3^n-2)^{1/2}}}\left(2+3e^{\frac{15t}{(34\cdot
3^n-2)^{1/2}}}\right)^{\frac{3^n-3}{6}}
\left(1+2e^{\frac{15t}{(34\cdot
3^n-2)^{1/2}}}\right)^{\frac{3^n+1}{2}},
$$
whose limit as $n\rightarrow \infty$ is $e^{\frac{t^2}{2}}$,
showing that the random variable is asymptotically normal.
\end{proof}

\section*{Acknowledgements}
We wish to express our deepest gratitude to Tullio
Ceccherini-Silberstein for useful comments and suggestions.

\end{document}